\documentclass[11pt]{amsart}
\pdfoutput=1
\usepackage[utf8]{inputenc}
\usepackage[T1]{fontenc}
\usepackage{lmodern}

\usepackage{amssymb}
\usepackage{mathtools}
\usepackage{amscd}
\usepackage[mathscr]{eucal}
\usepackage{enumitem}
\usepackage[hmargin=1in,vmargin=1in]{geometry}

\usepackage{graphicx}
\usepackage{color}

\usepackage[roman]{sublabel}

\usepackage[labelfont=bf,font=small]{caption}
\usepackage{subcaption}

\usepackage{cite}
\usepackage[foot]{amsaddr}
\usepackage{comment}

\definecolor{link_blue}{rgb}{0,0,0.97}
\definecolor{cite_red}{rgb}{0.7,0,0}
\usepackage[colorlinks=true,citecolor=cite_red,linkcolor=link_blue]{hyperref}

\usepackage{ottmacros-Gaussian-dSDEs-v11}

\DeclareMathOperator{\Cov}{Cov}

\DeclareMathOperator{\Tube}{Tube}
\DeclareMathOperator{\Grid}{Grid}

\theoremstyle{plain}
\newtheorem{theorem}{Theorem}[section]
\newtheorem{corollary}[theorem]{Corollary}

\newtheorem{proposition}[theorem]{Proposition}

\theoremstyle{definition}
\newtheorem{definition}[theorem]{Definition}
\newtheorem{notation}[theorem]{Notation}

\begin{document}

\title{Large deviations for Gaussian diffusions with delay}

\author{Robert Azencott, Brett Geiger and William Ott}
\address{University of Houston}




\date{\today}

\begin{abstract}
Dynamical systems driven by nonlinear delay SDEs with small noise can exhibit important rare events on long timescales.  When there is no delay, classical large deviations theory quantifies rare events such as escapes from metastable fixed points.  Near such fixed points, one can approximate nonlinear delay SDEs by linear delay SDEs.  Here, we develop a fully explicit large deviations framework for (necessarily Gaussian) processes $X_t$ driven by linear delay SDEs with small diffusion coefficients.  Our approach enables fast numerical computation of the action functional controlling rare events for $X_t$ and of the most likely paths transiting from $X_0 = p$ to $X_T=q$.  Via linear noise local approximations, we can then compute most likely routes of escape from metastable states for nonlinear delay SDEs.  We apply our methodology to the detailed dynamics of a genetic regulatory circuit, namely the co-repressive toggle switch, which may be described by a nonlinear chemical Langevin SDE with delay.
\end{abstract}

\maketitle

\thispagestyle{empty}

\section{Introduction}
\label{s:intro}

Dynamical processes are often influenced by small random fluctuations acting on a variety of spatiotemporal scales.
Small noise can dramatically affect the underlying deterministic dynamics by transforming stable states into metastable states and giving positive probability to rare events of high interest, such as excursions away from metastable states or transitions between metastable states.
These rare events play important functional roles in a wide range of applied settings, including genetic circuits~\cite{Eldar-Elowitz-Functional-2010},  molecular dynamics, turbulent flows~\cite{Bouchet-2014-Langevin}, and other systems with multiple timescales~\cite{Bouchet-2016-fast-slow}.

The main goal of this paper is to present an {\bfseries \itshape explicit} computational and theoretical large deviations analysis of rare events for Gaussian diffusion processes with {\bfseries \itshape delays}.
We are motivated in part by the importance of delay for the dynamics of genetic regulatory circuits. Indeed, we apply our approach to  a {\bfseries \itshape bistable genetic switch} driven by a delay stochastic differential equation (delay SDE) of Langevin type.

Consider a family of random processes $X^{\vep }(t) \in \mbb{R}^d$ indexed by a small parameter $\vep >0$  and driven by the following generic small-noise SDE with drift $b$, diffusion $\sqrt{\vep } \sigma $, and {\bfseries \itshape no delays}:
\begin{equation*}
\mrm{d} X^{\vep } (t) = b(X^{\vep } (t)) \, \mrm{d} t + \sqrt{\vep } \sigma (X^{\vep } (t)) \, \mrm{d} W(t).
\end{equation*}
Large deviations theory for SDEs of this form was developed by Freidlin and Wentzell~\cite{Freidlin-Wentzell-2012}.
Freidlin-Wentzell theory estimates the probability that the process $X^{\vep } (t)$ lies within a small tube around any given continuous path $\psi \in C([0,T])$ taking values in $\mbb{R}^{d}$ in terms of the {\bfseries \itshape action} $S_{T} (\psi )$ of $\psi $:
\begin{equation*}
\mbb{P}_{x} \vset{ \sup_{0 \leqs t \leqs T} \vnorm{ X^{\vep } (t) - \psi (t) }_{\mbb{R}^{d}} \leqs \delta } \approx \exp \left( - \vep^{-1} S_{T} (\psi ) \right).
\end{equation*}
Here $\mbb{P}_{x}$ denotes probability conditioned on $X^{\vep } (0) = x$ and we assume that $\psi (0) = x$.

The Freidlin-Wentzell action functional $0 \leqs S_T(\psi) \leqs \infty$ was originally defined for uniformly bounded coefficients $b,\sigma$ and uniformly elliptic $\sigma \sigma^*$ by an explicit  time integral involving $b(\psi_t)$, $\sigma \sigma^*(\psi_t)^{-1}$, and $\psi'_t$. These remarkable  results were widely extended by S. Varadhan \cite{Azencott-F-V} to arbitrary sets of trajectories and by R. Azencott \cite{Azencott-F-V} to hypoelliptic diffusions with unbounded coefficients. Numerous extensions and applications to broad classes of stochastic processes have been published by D. Stroock, R. Ellis, A. Dembo, O. Zeitouni, G. Dupuis, and many others.
For SDEs with delays, large deviations principles have been established or reasonably justified under  a variety of hypotheses ~\cite{Mohammed-Zhang-ld-2006, Mo-Luo-ld-2013, Bao-Yuan-ld-2015, Chiarini-Fischer-ld-2014, Kushner-ld-2010, Puhalskii-degenerate-ld-2004, Gadat-Panloup-Pellegrini-2013}. 

For fixed time $T$ and points $p,q$ in the state space, the path $\hat{\psi }$ that minimizes  $S_{T}(\psi)$ (under the constraints $\psi (0) = p$ and $\psi (T) = q$) is the most likely transition path starting at $p$ and reaching $q$ at time $T$. A second minimization over $T$ provides the most likely transition path from $p$ to $q$ and the energy $V(p,q)$ associated with this optimal path.
Often called the quasi-potential, $V$ is central to the quantification of large deviations on long timescales~\cite{Freidlin-Wentzell-2012}.

A computational framework has been developed for the application of Freidlin-Wentzell theory to systems with no delays.
This framework includes the minimum action method~\cite{E-Ren-EVE-2004-MAM}, an extension called the geometric minimum action method that synthesizes the minimum action method and the string method~\cite{Heymann-EVE-2008-gMAM}, as well as variants of these approaches (see {\itshape e.g.}~\cite{Li-Li-Zhou-constrained-MAM-2016, Lindley-Schwartz-2013-IAMM}).

For nonlinear delay SDEs, it is possible to compute a linear noise approximation~\cite{Brett-Galla-PRL-2013} that is valid in a neighborhood of a given metastable state.
Since linear noise approximations are Gaussian diffusions with delays, we have deliberately focused the present paper on Gaussian diffusions with delays.
For such diffusions, we rigorously develop and implement a {\itshape \bfseries fully explicit} large deviation framework, enabling fast numerical computation of optimal transition paths and the quasi-potential.
Our methodology does not require the numerical solution of Hamilton-Jacobi equations, a significant positive since Hamilton-Jacobi equations are computationally costly in even moderately-high spatial dimension.

We thus center our study on the It\^{o} delay SDE
\begin{equation}
\label{e:dSDE-linear}
\left\{
\begin{aligned}
\mrm{d} X^{\vep }_{t} &= (a + B X^{\vep }_{t} + C X^{\vep }_{t - \tau }) \, \mrm{d} t + \varepsilon \Sigma \, \mrm{d} W_{t},
\\
X^{\vep }_{t} &= \gamma (t) \text{ for } t \in [- \tau , 0].
\end{aligned}
\right.
\end{equation}
Here $X^{\vep }_{t} \in \mbb{R}^{d}$, $t$ denotes time, $\tau \geqs 0$ is the delay, $a \in \mbb{R}^{d}$, $B$ and $C$ are real $d \times d$ matrices, $W_{t}$ denotes standard $n$-dimensional Brownian motion, $\Sigma \in \mbb{R}^{d \times n}$ denotes the diffusion matrix, and $\varepsilon > 0$ is a small noise parameter.
The initial history of the process is given by the Lipschitz continuous curve $\gamma : [- \tau , 0] \to \mbb{R}^{d}$.
We work with a single fixed delay to simplify the presentation – all of our results apply just as well to multiple fixed delays and to delays distributed over a finite time interval.

The Gaussian diffusion~\eqref{e:dSDE-linear} arises via linear noise approximation of nonlinear delay SDEs near metastable states in the following way.
Suppose the nonlinear delay SDE
\begin{equation*}
\mrm{d} x_{t} = f(x(t),x(t - \tau )) \, \mrm{d}t + \vep g(x(t),x(t - \tau )) \, \mrm{d} W_{t}
\end{equation*}
has a metastable state $z$; that is, $z$ is a stable fixed point of the deterministic limit ODE
\begin{equation*}
\mrm{d} x_{t} = f(x(t),x(t - \tau )) \, \mrm{d}t.
\end{equation*}
Writing $x(t) = z + \xi (t)$ and expanding $f$ and $g$ around $z$ yields the linear noise approximation
\begin{equation*}
\label{eq:GaussApprox}
\mrm{d} \xi_{t} = \left[ D_{1} f (z, z) \xi (t) + D_{2} f (z, z) \xi (t - \tau ) \right] \mrm{d} t + \vep g(z, z) \, \mrm{d} W_{t},
\end{equation*}
where $D_{1}$ and $D_{2}$ denote differentiation with respect to the first and second sets of $d$ arguments, respectively.
This is~\eqref{e:dSDE-linear} with $a = \mbold{0}$, $ B= D_{1} f (z, z)$, $C = D_{2} f (z, z)$, and $\Sigma = g(z,z)$.

We demonstrate the utility of our approach by computing optimal escape trajectories for the co-repressive toggle switch, a bistable genetic circuit driven by a nonlinear delay Langevin equation.

The paper is organized as follows.
In Section~\ref{s:theory}, we review the theory of large deviations for Gaussian processes and present optimal transition path theory for~\eqref{e:dSDE-linear}.
We detail our numerical implementation of this theory in Section~\ref{s:implementation}.
Section~\ref{s:LNAs} discusses the general idea of linear noise approximation.
We present our computational study of a bistable genetic toggle switch in Section~\ref{s:toggle}.

\section{Theory}
\label{s:theory}

In this section we develop a rigorous large deviations framework for~\eqref{e:dSDE-linear}.

\subsection{Outline}
\label{ss:outline}

For brevity,  we will often omit the superscript $\vep$, writing $X_t$ instead of $X_{t}^{\vep}$. 
We first show that the process $X_{t}$  driven by~\eqref{e:dSDE-linear} is in fact a Gaussian process (Section~\ref{ss:Gaussian}). This is expected since~\eqref{e:dSDE-linear} is linear, but not obvious because of the presence of delay.
Since $X_{t}$ is a Gaussian process, it is completely determined by its mean $m(t) = \mbb{E} [X_{t}]$ and covariance matrices $\rho(s,t)= \mbb{E} [X_{s} X_{t}^{*}] - m(s) m(t)^{*}$.
Here $\ast $ denotes matrix transpose.
We derive and analytically solve delay ODEs verified by the mean and covariance matrices of $X_t$ in Sections~\ref{ss:stats}--\ref{ss:covariance-solution}. 

We center $X_{t}$ by writing $X_{t} = \mbb{E} [X_{t}] + \vep Z_{t}$. The probability distribution $\nu$ of $Z_{t}$ is a centered Gaussian measure on the space of continuous paths $f$ starting at $\mbold{0}$. To apply large deviations theory to the paths of $Z_t$ and $X_t$ as $\vep \to 0$, one needs to compute the action functional (or Cramer transform), $f \to \lambda(f)$, of $\nu$. Classical large deviations results for centered Gaussian measures on Banach spaces express $\lambda(f)$ in terms of the integral operator determined by  the matrix-valued covariance function $\rho(s,t)$ of $Z_{t}$ (see Sections~\ref{ss:general-ld-framework}--\ref{ss:Gaussian-Process}). Here we derive explicit formulas and implementable computational schemes that allow us to numerically evaluate $\lambda(f)$.
We complete this program by explicitly deriving, as $\vep \to 0$, the most likely transition path of $X_t$ between two points $p$ and $q$. 
We achieve this by minimizing $\lambda(f)$ under appropriate constraints.

We finish the outline by introducing notation that will be used throughout Section~\ref{s:theory}.

\begin{notation}
\label{E-and-H}
For a matrix $M$ and vector $v \in \mbb{R}^{d}$, let $\norm{M}$ and $\norm{v}_{\mbb{R}^{d}}$ denote matrix norm and Euclidean norm, respectively.
The scalar product of vectors $v, w \in \mbb{R}^{d}$ is denoted $\langle v, w \rangle_{\mbb{R}^{d}}$.

Let $H = L^2([0, T])$ be the Hilbert space of $\mbb{R}^d$-valued measurable functions $f$ on $[0,T]$ such that $\vnorm{f(t)}_{\mbb{R}^{d}}$ is square-integrable.  The Banach space of $\mbb{R}^d$-valued continuous functions $f$ on $[0, T]$ is denoted $C([0,T])$.
Let $E = C_{\mbold{0}} ([0,T])$ be the Banach subspace of all $f \in C([0,T])$ such that $f(0)= \mbold{0}$. 
Endow these three spaces with their Borel $\sigma$-algebras. 
\end{notation}

\subsection{The solution of~\eqref{e:dSDE-linear} is a Gaussian process}
\label{ss:Gaussian}

\begin{proposition}
\label{p:Gaussian-sol}
The delay SDE ~\eqref{e:dSDE-linear} has a unique strong solution $X_{t} \in \mbb{R}^d$. The process $X_{t}$ is Gaussian with almost surely continuous paths, and hence has a surely continuous version. 
\end{proposition}
\begin{proof}[Proof of Proposition~\ref{p:Gaussian-sol}] The existence of a unique strong solution $X_{t}$ is classical (see \textit{e.g.}~\cite{Mao-Stochastic-1997}).
To prove that $X_{t}$ is Gaussian, we consider Euler-Maruyama discretizations~\cite{Higham-2001} of $X_t$. For positive integers $N$, let $\Delta = \tau / N$ denote timestep size. 
The Euler-Maruyama approximate solution $Y^{(\Delta )}_{t}$ to~\eqref{e:dSDE-linear} is defined first at nonnegative integer multiples of $\Delta $ by
\begin{equation*}
Y^{(\Delta )}_{(k+1) \Delta } = Y^{(\Delta )}_{k \Delta } + (a + B Y^{(\Delta )}_{k \Delta } + C Y^{(\Delta )}_{k \Delta - \tau }) \Delta + \varepsilon \Sigma (W_{(k+1) \Delta } - W_{k \Delta }).
\end{equation*}
Then $Y^{(\Delta )}_{t}$ is extended to $[0,T]$ by linear interpolation. The convergence of this discretization scheme is well-known (see \textit{e.g.}~\cite{Baker-Buckwar-2000, Mao-Sabanis-2003}) and Theorem~2.1 of~\cite{Mao-Sabanis-2003} yields 
\begin{equation*}
\lim_{\Delta \to 0} \mbb{E} \left[ \sup_{0 \leqs t \leqs T} \vnorm{Y^{(\Delta )}_{t} - X_{t}}_{\mbb{R}^{d}}^{2} \right] = 0.
\end{equation*}
Since $Y^{(\Delta )}_{t}$ is a Gaussian process by construction, this $L^{2}$-convergence implies that $X_{t}$ is a Gaussian process as well. 
The expected values $\mbb{E} [X_{t}]$ are then finite and clearly bounded due to the delay SDE driving $X_t$. Again using the delay SDE, this implies that the covariance matrix $\rho(s,t)$ of $X_s$ and $X_t$ remains bounded by a constant multiple of $|t-s|$. A classical result of Fernique for Gaussian processes implies then that $X_t$ is almost surely continuous. 
\end{proof}

\subsection{Delay ODE for  the mean of $X_{t}$}
\label{ss:stats}
Writing~\eqref{e:dSDE-linear} in integral form, we have
\begin{equation}
\label{e:dSDE-linear-int}
X_t = X_0 + \int_0^t(a + B X_z + C X_{z-\tau}) \, \mrm{d}z + \varepsilon \Sigma W_{t}.
\end{equation}
Taking the expectation of~\eqref{e:dSDE-linear-int} and applying Fubini gives
\begin{equation*}
m(t) = m(0) + \int_0^t(a + B m(z) + C m(z-\tau)) \, \mrm{d}z,
\end{equation*}
or, in differential form, a delay ODE for $m(t)$:
\begin{equation}
\label{eq:mean}
\left\{
\begin{aligned}
m'(t) &= a + B m(t) + C m(t - \tau ),
\\
m(t) &= \gamma (t) \text{ for } t \in [-\tau , 0].
\end{aligned}
\right.
\end{equation}
\subsection{The centered Gaussian process $Z_{t}$}
The process $X_t$ is clearly not centered in general. The centered process $Z_t \in \mbb{R}^d$ defined by $X_t = m(t) + \varepsilon Z_t$ is a centered Gaussian diffusion with delay. Since $X_t$ verifies~\eqref{e:dSDE-linear} and $m(t)$ verifies~\eqref{eq:mean}, elementary algebra shows that $Z_t$ verifies the delay SDE
\begin{equation}
\label{e:dSDE-linear-centered}
\left\{
\begin{aligned}
\mrm{d} Z_{t} &= (B Z_{t} + C Z_{t - \tau }) \, \mrm{d} t + \Sigma \, \mrm{d} W_{t},
\\
Z_{t} &= \mbold{0} \text{ for } t \in [- \tau , 0].
\end{aligned}
\right.
\end{equation}
Note that this delay SDE \emph{does not depend on} $\vep$. The same is then true for $Z_t$. This is a crucial point further on because our key large deviations estimates for $\vep \to 0$ will be stated in path space for the “small” centered Gaussian process $\vep Z_t$. 
As we will see, our large deviations computations will ultimately involve the deterministic mean path $m(t)$ of $X_{t}$ and the covariance  function $\rho(s,t)$ of the process $Z_t$. 
\subsection{Delay ODEs for  the covariances  of $Z_{t}$}
\label{ss:covstats}

We now find delay ODEs for the covariance function of $Z_t$.
Let
\begin{equation} \label{rho}
\rho(s,t) = \mbb{E}[Z_s Z_t^{*}] 
\end{equation}
be the covariance matrix of $Z_s$ and $Z_t$, where superscript $*$ denotes matrix transpose.
Since the history of $Z_{t}$ anterior to $t=0$ is deterministic, $\rho(s,t) = 0$ when either $s$ or $t$ is in $[-\tau , 0]$.
Fix $t \in [0,T]$, and let $s$ vary. We have
\begin{gather}
\mbb{E}[Z_s Z_t^{*}] = \int_0^s  \left( B\mbb{E}[Z_u Z_t^{*}] + C \mbb{E}[Z_{u-\tau}Z_t^{*}] \right) \mrm{d}u + \Sigma \mbb{E}[W_{s}Z_t^{*}].
\notag
\end{gather}
We thus obtain
\begin{gather}
\begin{aligned}
\rho(s,t) &= \int_0^s (B \rho (u,t) + C \rho (u-\tau, t)) \, \mrm{d}u +  \Sigma \mbb{E}[W_{s}Z_t^{*}].
\\
\end{aligned}
\label{e:rho-integral}
\end{gather}
Let $G(s,t) = \mbb{E}[W_{s}Z_t^{*}].$
Differentiating $\rho(s,t) $ with respect to $s$ gives
\begin{equation}
\label{eq:cov}
\frac{\partial \rho}{\partial s}(s,t) = B \rho (s,t) + C \rho (s-\tau,t) + \Sigma \frac{\partial G}{\partial s}(s,t),
\end{equation}
which is a first-order delay ODE in $s$  for each fixed $t$. 
To close~\eqref{eq:cov}, we compute a differential equation for $\frac{\partial G}{\partial s}(s,t)$. 
Proceeding as just done for~\eqref{e:rho-integral}, one checks that the function $G(s,t)$ satisfies the delay ODE
\begin{equation}
\label{eq:Gfunc}
\frac{\partial G}{\partial t}(s,t) = 
\begin{cases}
G(s,t)B^{*} + G(s,t-\tau)C^{*} +  \Sigma^{*} & (t\leqs s),\\
G(s,t)B^{*} + G(s,t-\tau)C^{*} & (t>s),
\end{cases}
\end{equation}
where $G(s,t) = 0$ for $t\in [-\tau,0].$ Let $H(x)$ denote the Heaviside function 
\begin{equation*}
H(x) = 
\begin{cases}
0, &\text{if } x < 0;\\
1, &\text{if } x \geqs 0.
\end{cases}
\end{equation*}
We can rewrite~\eqref{eq:Gfunc} as
\begin{equation}
\label{e:Gfunc-heavy}
\frac{\partial G}{\partial t}(s,t) = G(s,t)B^{*} + G(s,t - \tau)C^{*} + \Sigma^{*} H(s - t).
\end{equation}
Note that the  partial derivative of the Heaviside distribution  $H(s-t)$ is classically given by
$$
\frac{\partial H}{\partial s}(s-t) = \delta(s-t),
$$
where the distribution $\delta$ is the Dirac point mass concentrated at zero.
By definition of $G(s,t)$ and by~\eqref{e:Gfunc-heavy}, the function  $G(s,t)$ is continuous for all $s$ and $t$ and differentiable in $s$ and $t$ for  $s \neq t$.  
For $s \neq t$, we write
\begin{equation}
\label{e:Fst}
F(s,t) = \frac{\partial G}{\partial s}(s,t),
\end{equation} 
and observe that $F$ verifies the initial condition $F(s,t) = 0$ for  $s \neq t$ and $t \in [-\tau ,0]$.

Differentiating~\eqref{e:Gfunc-heavy} in $s$ for $s\neq t$ and switching the order of  partial derivatives yields a linear delay ODE in $t >0$ for $F(s,t)$, namely
\begin{equation}
\label{eq:Gfunc-simplified}
\frac{\partial F}{\partial t}(s,t) = F(s,t)B^{*} + F(s,t - \tau)C^{*} + \Sigma^{*}\delta(s - t),
\end{equation}
with initial condition $F(s,t) = 0$ for all $t\in [-\tau,0]$.

Once $F(s,t)$ is determined, the covariance $\rho(s,t)$ for each fixed $t\in [0,T]$ will be computed by solving  the delay ODE
\begin{equation}
\label{e:cov-simplified}
\frac{\partial \rho}{\partial s}(s,t) = B \rho (s, t) + C \rho (s-\tau, t) +  \Sigma F(s,t).
\end{equation}
We now describe how to successively solve the delay ODEs driving  $m(t), F(s,t),$ and $\rho (s,t)$. 

\subsection{Analytical solution of the delay ODE verified by the mean}

First-order delay ODEs can be analytically solved by a natural stepwise approach, sometimes called the ``method of steps,'' a terminology which we will avoid since it is has a different meaning in classical numerical analysis. 
The basic idea is to convert each one of our delay ODEs into a finite sequence of nonhomogeneous ODEs in which the delay terms successively become known terms. 

Consider first the delay ODE (\ref{eq:mean}) for $m(t)$ with $t \in [-\tau,T]$. The delay term $C m(t-\tau)$ is unknown for $t \in (\tau,T]$ but is known for $t\in [0,\tau]$. So we can solve the delay ODE (analytically or numerically) on the interval $[0,\tau]$ as a linear nonhomogeneous first-order ODE. Then, for $t \in [\tau,2\tau]$, the delay term in the delay ODE has just been computed, so this delay ODE once again becomes a linear nonhomogeneous first-order ODE.
\begin{notation}
\label{Jk} 
To get a solution on the whole of $[0,T]$, one successively solves the delay ODE on closed intervals $J_k = [k \tau, (k+1) \tau]$ with $k = -1, 0, 1, 2, \dots, N$. Here $N= \lfloor{\frac{T}{\tau}}\rfloor$ is the largest integer inferior or equal to $\frac{T}{\tau}$, and $J_N = \left[ N \tau, T \right].$
\end{notation}
We now compute the explicit solution for the mean $m(t)$ on $[-\tau,T].$ Let $m_k$ denote the solution of~\eqref{eq:mean} on the interval $J_k$. For $k = -1$, we have $m_{-1} = \gamma$ on $J_{-1}$. For $k = 0$, $m_0$ is the solution of~\eqref{eq:mean} on $J_0$:
\begin{equation*}
\left\{
\begin{aligned}
m_0'(t) &= a + B m_0(t) + C \gamma(t - \tau ),
\\
m_0(0) &= \gamma(0).
\end{aligned}
\right.
\end{equation*}
We thus have 
\begin{equation*}
m_0(t) = e^{t B} \int_{0}^t \, e^{- u B} \left( a + C \gamma(u - \tau) \right) \mrm{d}u + e^{t B} \gamma(0).
\end{equation*}
Similarly, given $m_{k-1}$ on $J_{k-1}$, $m_k$ has the form
\begin{equation*}
m_k(t) = e^{(t - k \tau) B} \int_{k \tau}^t \, e^{- (u - k \tau) B} \left(a + C m_{k-1}(u - \tau) \right) \, \mrm{d}u + e^{(t - k \tau) B} m_{k-1}(k \tau).
\end{equation*}
Piecing together the $m_k$ yields the full solution $m$ on all of $[-\tau,T]$.

Note that many characteristics of the given history function $\gamma$, such as continuity, differentiability, discontinuities, etc.,  will essentially propagate through to the full solution $m(t)$. More precisely, if $\gamma$ is of class $ C^q$ for some integer $q \geqs 0$, then  $m(t)$ will be of class $q+1$ for all positive $t$ except possibly at integer multiples of $\tau.$ Since we assume here that $\gamma$ is Lipschitz continuous, $m(t)$ will be differentiable except possibly at integer multiples of $\tau$.

\subsection{Analytical solutions of the delay ODEs verified by $F(s,t)$ and $\rho(s,t)$}
\label{ss:covariance-solution}
We can extend the preceding method to the delay ODE in $t$ verified by $F(s,t)$ for each fixed $s$ and then to the delay ODE in $s$ verified by $\rho(s,t)$. We first focus on  $F(s,t)$. 
Fix $s\in [0,T]$. Due to  the delay ODE~\eqref{eq:Gfunc-simplified},  the distribution $\phi_s$ defined on $\mbb{R}^+ $ by  
$$
\phi_s(t) = F(s,t) + \Sigma^{*} H(s-t)
$$ 
clearly verifies the delay ODE
\begin{equation}
\label{eq:phi}
\frac{\partial \phi_s}{\partial t}(t) - \phi_s(t) B^{*} - \phi_s (t - \tau)C^{*} = -\Sigma^{*} H(s-t) B^{*}  - \Sigma^{*} H(s-t+\tau) C^{*}  
\end{equation}
with initial condition    
\begin{equation*}
\phi_s(t) = \Sigma^{*} H(s-t) = \Sigma^{*} 
\end{equation*}
for all $t \in [-\tau,0]$ and $s > 0$.
Note that on $[-\tau, 0]$, this initial condition is constant and hence continuous. 
For each fixed $s$, we write the right side of equation~\eqref{eq:phi} as
\begin{equation*}
\theta_s (t) = -\Sigma^{*} H(s-t) B^{*}  - \Sigma^{*} H(s-t+\tau) C^{*}.  
\end{equation*}
Observe that for $t \geqs 0$, $\theta_{s}$ is bounded in $t$ (uniformly in $s$) and continuous in $t$ except for the two points $t=s$ and $t = s + \tau$. 
As was done above for $m(t)$, one can perform the iterative  analysis of the delay ODE~\pref{eq:phi} on successive time intervals $J_k = [ k \tau, (k+1) \tau]$. 
Since both the initial condition and the right side $\theta_s$ are known, the $k^{\text{th}}$ step of this iterative construction amounts to solving a first-order linear ODE with constant coefficients and known right side.
So this construction is essentially stepwise explicit  and proves by induction on $k$ that the distribution $\phi_s(t)$ is actually a bounded function of $t$ which is differentiable except maybe at points of the form $t= k \tau$ and $t= s+k \tau$.

For each $s \geqs 0$, once the full solution $\phi_s$ has been constructed on $[-\tau, T]$ as just outlined, we immediately obtain $F(s,t) = \phi_s(t) - \Sigma^{*} H(s-t)$.

At this stage, $F(s,t)$ is theoretically known over $t \in [-\tau, T]$ for each $s \geqs 0$ and can be plugged into the delay ODE~\eqref{eq:cov} in $s$ verified by $\rho(s,t)$ for each fixed $t$, with initial conditions $\rho(s,t) =0 $ for $(s,t) \in [-\tau,0] \times [-\tau, 0]$. 
For each fixed  $t\in [0,T]$, this delay ODE for $s \mapsto \rho(s,t)$  can again be solved iteratively on the successive time intervals $J_k$. 

The preceding analysis is easy to implement numerically to solve the three types of delay ODEs involved. Each reduction to a succession of roughly $T/\tau$ linear ODEs enables the use of classical numerical schemes to compute $m(t)$ and $\rho(s,t)$. We have used the now fairly standard numerical approach of \cite{Bellen-Zennaro-2003}. Our numerical implementation is presented in Section~\ref{s:implementation}.

We have focused on $m(t)$ and $\rho(s,t)$ because these two functions essentially determine the rate functional of large deviations theory for the Gaussian diffusion with delay $X_t^\vep$.

\subsection{General large deviations framework}
\label{ss:general-ld-framework}

We present, without proofs, a brief overview of large deviations theory for Gaussian measures and processes (refer to Chapter $6$ in~\cite{Azencott-F-V} for proofs of theorems).   We will then apply these principles to Gaussian diffusions with delay. 

The following notations and definitions will be used throughout this section.
\begin{itemize}
\item $E$ is any separable Banach space with dual space $E^{*}$ and duality pairing $\langle v,x \rangle \defas v(x)$ for  $v\in E^{*}$ and $x \in E$. 
\item $\nu$ is any probability on the Borel $\sigma$-algebra $\mathcal{B}(E)$. 
\item For $v\in E^{*}$, the image probability $v(\nu)$  is defined  on $\mbb{R}$ by   $[v(\nu)](A) := \nu(v^{-1}(A))$ for all Borel subsets $A$ of $\mbb{R}$.
\item $\nu$ is called \textit{centered} iff $v(\nu)$ is centered for all $v\in E^{*}$.
\item $\nu$ is called \textit{Gaussian} iff for all  $v \in E^{*}$, the image probability $v(\nu)$ is a Gaussian distribution on $\mbb{R}$. 
\end{itemize}
The Laplace transform, $\hat{\nu}(v)$, is defined as follows for $v \in E^{*}$: 
\begin{equation*}
\hat{\nu}(v) = \int_{E} e^{\langle v,x\rangle} \, \mrm{d}\nu(x).
\end{equation*}
A classical result asserts that when $\nu$ is a Gaussian probability measure on a separable Banach space E, then its Laplace transform $\hat{\nu}(v)$ is finite for all $v \in E^{*}.$

We first recall key large deviations results for generic Borel probabilities $\nu$ on separable Banach spaces $E$. Later on below, $E$ will be the space $C_{\mbold{0}} ([0,T])$ and $\nu$ will be Gaussian. 
Probabilities of rare events for the empirical mean of independent random vectors with identical distribution $\nu$ can be estimated via  a non-negative functional, the Cramer transform $\lambda$ of $\nu$, defined as follows (see Theorem 3.2.1 in~\cite{Azencott-F-V}). 
\begin{definition}
\label{d:Cramer}
Assume $\nu$ has finite Laplace transform $\hat{\nu}$.  The \textit{Cramer transform} $\lambda$ of $\nu$, also called the large deviations \textit{rate functional} of $\nu$, is defined for $x\in E$ by
\begin{equation*}
\lambda(x) = \sup_{v \in E^{*}} \left[ \langle v,x\rangle - \log \hat{\nu}(v) \right].
\end{equation*}
Note that $0 \leqs \lambda(x) \leqs + \infty$.
The \textit{Cramer set functional} $\Lambda(A)$ is then defined for all $A \subset E$ by
\begin{equation*}
\Lambda(A) = \inf_{x \in A} \lambda(x).
\end{equation*}
\end{definition}

When $E$ is a Banach space of continuous paths $f$, the value $\lambda (f)$ of the Cramer transform can be viewed as the ``energy'' of the path $f$. For instance, if $\nu$ is the probability distribution of the Brownian motion $W_t \in \mbb{R}^d$ in its path space, the Cramer transform is the kinetic energy  $\lambda(f) = \frac{1}{2}\vnorm{f '}_{L^2}^2$ (see Proposition 6.3.8 in~\cite{Azencott-F-V}). 
The set functional $\Lambda(A)$ quantifies the probability of rare events in path space through classical large deviations inequalities, which we now recall. 

\subsection{Gaussian framework and associated Hilbert space}\label{ss:general-Gaussian-framework}
Here, $Z$ is the random path of a centered continuous Gaussian process $Z_t$ driven by the delay SDE~\eqref{e:dSDE-linear-centered}. The probability distribution $\nu$ of $Z$ is a centered Gaussian probability on the separable Banach space $E = C_{\mbold{0}} ([0,T])$. So we now focus on Gaussian probabilities on separable Banach spaces.
We first state key large deviations inequalities due to S. Varadhan.
\begin{theorem}[see Theorem 6.1.6 in~\cite{Azencott-F-V}]\label{thrm:GeneralLargeDev}
Let $\nu$ be a centered Gaussian probability measure on a separable Banach space $E$. Let $Z $ be an E-valued  random variable with probability distribution $\nu$. Let $\Lambda$ be the Cramer set functional of $\nu$. For every Borel subset $A$ of $E $ one has 
\begin{equation}
\label{eq:varadhan}
-\Lambda(A^{\circ}) \leqs \liminf_{\varepsilon \to 0} \varepsilon^2\log \mbb{P} (\varepsilon Z\in A) \leqs \limsup_{\varepsilon \to 0} \varepsilon^2 \log \mbb{P} (\varepsilon Z\in A) \leqs -\Lambda(\bar{A}).
\end{equation}
where $A^{\circ}$ and $\bar{A}$ denote the interior and the closure of $A$, respectively.
\end{theorem}
Whenever $\Lambda(A^{\circ}) = \Lambda(\bar{A})$, then the lower and upper limits in \eqref{eq:varadhan} are equal and
\begin{equation*}
\lim_{\varepsilon \to 0} \varepsilon^{2} \log \mbb{P}(\varepsilon Z \in A) =  -\Lambda(A).
\end{equation*}
In this case, for small $\varepsilon$ one has the rough estimate 
\begin{equation*}
\log \mbb{P} (\varepsilon Z\in A) \approx -\frac{\Lambda(A)}{\varepsilon^{2}}.
\end{equation*}
The equality $\Lambda(A^{\circ}) = \Lambda(\bar{A})$ holds, for example, when the Cramer transform is finite and continuous on $\bar{A}$, and $\bar{A}$ is the closure of $ A^{\circ}$.
\begin{definition}
In the Banach space context, the \emph{covariance kernel} $\Cov : E^{*} \times E^{*} \to \mbb{R}$ of $\nu$ is defined for all $v,w \in E^{*}$ by
\begin{equation} \label{COV}
\Cov (v,w) = \int_{E} \langle v, z \rangle \langle w, z \rangle \, \mrm{d} \nu (z).
\end{equation}
\end{definition}
Let $\Omega$ be the probability space $(E, \nu)$. The covariance kernel $\Cov(v,w)$ defines a linear embedding of $E^{*}$ into a Hilbert subspace $H$ of $L^{2}(\Omega)$ as follows. For each $v \in E^{*}$, define a Gaussian random variable $Y_{v}$ on $\Omega$ by $Y_v(z) = \langle v, z \rangle$ for all $z \in\Omega$. 
Let $H$ be the closure in $L^{2}(\Omega)$ of the vector space spanned by all the $Y_{v}$ with $v \in E^{*}$. The linear map $Y : E^{*} \to H$ defined by $Y(v) = Y_{v}$ is continuous with dense image in $H.$ The inner product in $H$  verifies 
\begin{equation*}
\langle Y_{v}, Y_{w} \rangle_{H} = \Cov(v,w). 
\end{equation*} 
Define a continuous linear operator $G : L^{2}(\Omega) \to E$ by 
\begin{equation*}
G(\eta) = \int_{E} z \, \eta(z)\, \mrm{d}\nu(z)
\end{equation*}
for all $\eta \in L^{2} (\Omega)$.
Then for all $v \in E^{*}$ and $\eta \in L^{2}(\Omega)$, one has
$$
\langle Y_{v}, \eta \rangle_{H} = \langle v, G(\eta) \rangle ,
$$
so that $G$ restricted to $H$ is injective.
Within this classical Gaussian framework, one obtains a generic expression for the Cramer transform $\lambda$ of $\nu$ (see Theorem 6.1.5 in~\cite{Azencott-F-V}). 
\begin{theorem}\label{thrm:GeneralBanachCramer}
Let $\nu$ be a centered Gaussian probability on the separable Banach space $E.$ Let $H$ and $G: H \to E$  be, respectively, the Hilbert space and the continuous linear injection associated to $\nu$ by the preceding construction.  Then $G$ is a compact operator, and for $z \in E$, the Cramer transform $\lambda$ of $\nu$ is given by
\begin{equation*}
\lambda(z) = 
\begin{cases}
\frac{1}{2} \vnorm{G^{-1} z}_{H}^2, &\text{if } z \in G(H);\\
\infty, &\text{otherwise}.
\end{cases}
\end{equation*}
\end{theorem}
\subsection{Large deviations rate functional for continuous Gaussian processes}
\label{ss:Gaussian-Process}
For a centered continuous Gaussian process on $[0, T]$ with probability distribution $\nu$ in the path space $C_{\mbold{0}} ([0,T])$, the generic framework in Section~\ref{ss:general-Gaussian-framework} can be applied to $E = C_{\mbold{0}} ([0,T])$ with $H =  L^{2} ([0,T])$ and $E^*$ as the space of bounded Radon measures on $[0,T]$ to obtain a more explicit form of the operator $G$ in Theorem~\ref{thrm:GeneralBanachCramer}.  For a detailed explanation on the connection among the spaces $E, E^*,$ and $H$ in this context, see Section 6.3 in~\cite{Azencott-F-V}.  The rate functional of $\nu$, when viewed either on $H= L^{2}([0,T])$ or $E =  C_{\mbold{0}} ([0,T]),$ can be expressed via Proposition 6.3.7 and Lemma 6.3.6 in~\cite{Azencott-F-V}, which we reformulate as follows. 
\begin{proposition}
\label{p:General-ld-result}
Let $Z_{t} \in \mbb{R}^d $ be any centered continuous Gaussian process on $[0,T]$ with continuous matrix-valued covariance function $\rho(s,t).$ The random path $Z : t \to Z_t$ takes values in the Banach space $E = C_{\mbold{0}} ([0,T])$, and hence in the Hilbert space $H = L^{2} ([0,T])$. Since the inclusion $j : E \to H$ is continuous and injective, the probability distribution $\nu$ of $Z$ can be viewed either as a centered Gaussian measure on $E$ or $H$.
Recall that the self-adjoint covariance operator $R : H \to H$ of $Z_{t}$ is defined by 
\begin{equation} 
\label{covarop}
R f(t) = \int_{0}^{T} \rho (t,u) f(u) \, \mrm{d} u 
\end{equation}
for all $f \in H$.
Then $R$ is a compact linear operator with finite trace, and $R(H) \subset E$. Moreover, $R$ is semi-positive definite.
Let $V \subset H $ be the orthogonal complement in $H$ of the kernel of $R$. Call $S$ the restriction of $\sqrt{R}$ to $V$. Then $S : V \to H$ is injective and maps $V$ onto $\sqrt{R} (H)$.
Call $\lambda_H : H  \to [0, + \infty]$ the Cramer transform of $\nu$ on the Hilbert space $H$. Then for any $f \in H$, one has 
\begin{equation} \label{lambdaH}
\lambda_H (f) =
\begin{cases}
\frac{1}{2} \vnorm{S^{-1} f}^{2}_H, &\text{if } f \in \sqrt{R} (H);
\\
\infty , &\text{otherwise}.
\end{cases}
\end{equation}
When viewed as a probability on the Banach space $E = C_{\mbold{0}} ([0,T])$, the probability $\nu$ has a Cramer transform $\lambda_E$, and one has
\begin{equation} \label{lambdaE}
\lambda_E (f) = \lambda_H (f)
\end{equation}
for all $f \in E = C_{\mbold{0}} ([0,T])$.
\end{proposition}
\subsection{Large deviations rate functional for Gaussian diffusions with delay}\label{ss:path-theory}
In this section we compute explicitly the large deviations rate functional for the centered Gaussian process $Z_t \in \mbb{R}^d$ driven by the delay SDE~\eqref{e:dSDE-linear-centered}. We will adapt to this delay SDE a technique introduced by R. Azencott in \cite{Azencott} and \cite{Azencott-F-V} to study large deviations for hypoelliptic diffusions with unbounded smooth coefficients.
\begin{proposition}   
\label{supportZ}
Consider the centered Gaussian process $Z_t \in \mbb{R}^d $ driven by the delay SDE~\eqref{e:dSDE-linear-centered}, and call $Z = Z([0, T]) \in E = C_{\mbold{0}} ([0,T]) $ the random path $t \to Z_t$. Let $\nu$ be the probability distribution of $Z$ on the Banach space $E$.
Let $W = W([0, T]) \in E$ be the continuous random path of the Brownian motion $W_t \in \mbb{R}^d$ driving the delay SDE verified by $Z_t$. 
Let $\Sigma$ be the $d \times d$ matrix of diffusion coefficients in the delay SDE verified by $Z_t$. We now assume that $\Sigma$ has \emph{full rank} $d$.

Then, there is a bijective linear map $\Gamma$ from $E$ onto $E$ such that $\Gamma$ and $\Gamma^{-1}$ are both continuous and the random paths $Z$ and $W$ verify almost surely $Z= \Gamma (W)$. Moreover, $\Gamma$ is defined by iterating operators explicitly given in equation \eqref{map3} below. 
Let $\mathcal{H} \subset E = C_{\mbold{0}} ([0,T])$ be the dense subspace of all paths $g$ in $E$ such that $g’$ is in $H = L^2([0, T])$. The restriction of $\Gamma$ to $\mathcal{H}$ is a linear bijection onto $\mathcal{H}$. For $f \in \mathcal{H}$ the function $g = \Gamma^{-1}f $ is given by equation \eqref{inversegamma} below, and is in $\mathcal{H}$.

In the Banach space $E$, the support of the probability distribution $\nu$ of $Z$ is equal to $E$, and the only closed vector subspace $F$ of $E$ such that $\nu(F) =1$ is $F=E$. Similarly, in the Hilbert space $H$, the support of $\nu$ is equal to $H$, and the only closed subspace $F$ of $H$ such that $\nu(F) =1$ is $F=H$.
\end{proposition}
\begin{proof}
The centered process $Z_t$ is driven by the delay SDE
$$
\mrm{d}Z_t - B Z_t \, \mrm{d}t = C Z(t - \tau) \, \mrm{d}t + \Sigma \, \mrm{d} W_t .
$$
For $n = -1, 0, 1, \ldots, N$ and $N= \lfloor \frac {T}{\tau}\rfloor $, define as above the interval $J_n= [t_n, t_{n+1}] $ with $t_n= n \tau$ and $t_{N+1}= T.$
Extend any function in $E$ by giving it the value $0$ on $J_{-1} = [- \tau, 0]$. For any $g$ in $\mathcal{H} \subset E$, we now prove that there is  a unique $f \in \mathcal{H}$ , denoted $f = \Gamma g$, verifying $f=0$ on $J_{-1}$ and the following delay ODE:
\begin{equation} \label{map1}
f’(t) - B f(t) = C f(t - \tau) + \Sigma g’(t).
\end{equation}
To construct $f$ given $g$, we set $F(t) = e^{-t B} f(t)$. The linear delay ODE \eqref{map1} can be (uniquely) solved in $f$ successively on the intervals $J_n$ as seen earlier, and since $g’ \in H,$ the same recurrence on $n$ easily shows that $f’$ is in $H$. 
We now note that equation \eqref{map1} has the following equivalent integral form, valid for all $s,t \in [0,T]$ with $s \leqs t:$
\begin{equation} \label{integralmap}
e^{-t B} f(t) - e^{-s B} f(s) = \int_s^t \; e^{-u B} [ C f(u-\tau) + \Sigma g’(u) ]\, \mrm{d}u .
\end{equation}
After setting $g = f = 0$ on $J_{-1}$, equation \eqref{map1} can hence be successively solved on the intervals $J_{n}$ with $n = 1, \ldots, N$ by applying, for $t \in J_{n},$ the recurrence formula
\begin{equation} \label{map2}
e^{-t B} f(t) - e^{-t_n B} f(t_n) = 
\int_{t_n}^t \, e^{-u B} \, [ C f(u - \tau) + \Sigma g’(u) ]\, \mrm{d}u.
\end{equation}
Integration by parts of $ e^{-u B} \Sigma g’(u) \,\mrm{d}u $ in \eqref{map2} yields, for $n \geqs 1$ and  $t \in J_{n}$,
\begin{equation} 
\label{map3}
\begin{split}
e^{-t B} f(t) - e^{-t_n B} f(t_n) &= \int_{t_n}^t \, e^{-u B} \, C f(u - \tau)\,\mrm{d}u + e^{-t B} \Sigma g(t) -  e^{-t_n B} \Sigma g(t_n)
\\
&\quad {}+ \int_{t_n}^t \, B e^{-u B} \, \Sigma g(u)\,  \mrm{d}u.
\end{split}
\end{equation}
The new iterative formula \eqref{map3} does not involve $g’$ anymore and hence remains well-defined for all functions $g \in E = C_{\mbold{0}} ([0,T]).$ 
For each $g \in E$, denote $f = \Gamma g$ the continuous function $f$ determined iteratively on the intervals $J_{n}$ by the integral equations \eqref{map3}, initialized with $f = 0$ on $J_{-1}$. These equations show by recurrence on $n$ that $\Gamma $ is a continuous linear mapping of the Banach space $E$ into $E$. 
We now show that given any $f \in E$, one can construct a unique $g \in E$ such that $\Gamma g = f$. Fix $n \geqs 0$. Assume that for all $t \in [0, t_n]$, the value $g(t)$ is already known and verifies, for some constant $c_n $ independent of the function $f$,
$$
\vnorm{g(t)}_{\mbb{R}^{d}} \leqs c_n \vnorm{f}_E 
$$
for $t \in [0, t_n]$.
We then want to solve in $g$ the equation \eqref{map3} for $t \in J_n= [t_n, t_{n+1}]$. Define 
$$ 
k(t)= e^{-t B} \Sigma g(t) \;\; \text{and } \; K(t) = \int_{t_n}^t \; k(u) \, \mrm{d}u
$$
so that $g(t) = \Sigma^{-1} e^{t B} K’(t)$. 
Note that $k(t_n) = e^{-t_n B} \Sigma g(t_n)$ is known. For $t \in J_n$ define 
$$
F_n(t) = k(t_n) + e^{-t B} f(t) - e^{-t_n B} f(t_n) - \int_{t_n}^t \, e^{-u B} \, C f(u - \tau)\,\mrm{d}u.
$$
Let $\vnorm{\cdot}$ denote matrix norm. For $t \in J_n$, the expression of $F_n$ yields  
$$
\vnorm{F_n(t)}_{\mbb{R}^{d}}\leqs c_n \vnorm{f}_{E}\, e^{T \vnorm{B}}(\vnorm{\Sigma} +2 + T).
$$
Note that for $t \in J_0$, one has $t_0 = 0$ and $k(t_0) = f(t_0) = 0$ so that the preceding bound is valid with $c_0=1$.
Then for $t \in J_n,$ equation~\eqref{map3} becomes 
$$
K’(t) + B K(t) = F_n(t), 
$$
with $K(t_n) = 0$.
This linear ODE in $K$ has, for $t \in J_n$, a unique solution given by 
$$
K(t) = e^{-t B} \int_{t_n}^t \; e^{u B} F_n(u) \, \mrm{d}u.
$$
The bound on $\vnorm{F_n(u)}_{\mbb{R}^{d}}$ for $u \in J_n$ then implies 
$$
\vnorm{K(t)}_{\mbb{R}^{d}} \leqs c_n \vnorm{f}_{E} \cdot T \cdot e^{3 T \vnorm{B}}(\vnorm{\Sigma} +2 + T) 
$$
for $t \in J_n$. 
For $t \in J_n$, the function $g$ is now uniquely determined by 
$$
g(t) = \Sigma^{-1} e^{t B} K’(t) = \Sigma^{-1} e^{t B} [ - B K(t) + F_n(t) ].
$$
This implies $\vnorm{g(t)}_{\mbb{R}^{d}} \leqs \alpha  c_n \vnorm{f}_{E}$ for all $t \in J_n$, where the constant $\alpha$ is given by
$$
\alpha = \big\| \Sigma^{-1} \big\| (\vnorm{\Sigma} +2 + T) e^{4 T \vnorm{B}} (1+\vnorm{B} T), 
$$
as is easily deduced from the bounds on $\vnorm{K(t)}_{\mbb{R}^{d}}$ and $\vnorm{F_n(t)}_{\mbb{R}^{d}}$.
Hence, $g(t)$ is now known for all $t \in [0, t_{n+1}],$ and on this interval, one has $\vnorm{g(t)}_{\mbb{R}^{d}}\leqs c_{n+1} \norm{f}_E$, with $ c_{n+1} = (1 + \alpha) c_n$.
We know that $g(0) =0$, and that $c_0=1$, so we can start this construction of $g$ with $n= 0$ and $t \in J_0$, which will yield $g(t_1)$. 
We then proceed by recurrence on $n= 0, 1, \ldots, N$ as just indicated to uniquely determine $g$ on $[0, T]$ such that $\Gamma g = f$. We have also proved that the bound $\norm{g}_E \leqs c_N \norm{f}_E$ holds for all $f \in E$ for the fixed constant $c_N= (1+ \alpha)^N$.
This proves that the continuous linear map $\Gamma : E \to E$ has a continuous inverse $\Gamma^{-1} : E \to E$, which of course must be linear.
 
By construction, for any $g \in\mathcal{H}$ , the function $f = \Gamma g$ is also in $\mathcal{H}$  and verifies the delay ODE \eqref{map1}. Conversely, given any $f \in \mathcal{H}$, we know that there is a unique $g = \Gamma^{-1} f \in E$ such that $f = \Gamma g$. We now prove that $g$ must belong to $\mathcal{H}$. Indeed, we can define $g = \Gamma^{-1} f $ explicitly by $g(0) = 0$ and by
\begin{equation} \label{inversegamma}
\frac {d}{dt} \Gamma^{-1} f (t) = g’(t) = \Sigma^{-1} [ f’(t) - B f(t) - C f(t - \tau) ] 
\end{equation}
for all $t \in [0, T]$.
This relation clearly implies $g \in \mathcal{H}$ and $f = \Gamma g$. Hence, the restriction of  $\Gamma$ to $\mathcal{H}$ defines a linear bijection of $\mathcal{H}$ onto $\mathcal{H}$.

Let now $W \in E$ be the random path of the Brownian motion $W_{t}$ on $[0,T]$. Then $V= \Gamma W \in E$ is a continuous random path on $[0, T]$, starting at $V(0) = 0$. By equation \eqref{map3}, the path $V_t$ verifies, for $t \in J_{n}$, 
\begin{equation} \label{map4}
e^{-t B} V_t - e^{-t_n B} V_{t_n} = 
\int_{t_n}^t \, e^{-u B} \, C V_{u - \tau} \, \mrm{d}u + A_n(t),
\end{equation}
where we have set 
\begin{equation} \label{map5}
A_n(t) = e^{-t B} \Sigma W_t - e^{-t_n B} \Sigma W_{t_n} + \int_{t_n}^t \, B e^{-u B} \Sigma\,  W_u \, \mrm{d} u.
\end{equation}
For $t \in J_{n}$, Ito calculus enables the integration by parts of the last integral in \eqref{map5} to express $A_n(t)$ as a stochastic integral,
$$
A_n(t) = \int_{t_n}^t \, e^{-u B}\Sigma  \,   \mrm{d} W_u .
$$
Hence for $t \in J_{n}$, equation \eqref{map4} becomes
$$
e^{-t B} V_t - e^{-t_n B} V_{t_n} = \int_{t_n}^t \, e^{-u B} \, C V_{u - \tau} \,\mrm{d}u + \int_{t_n}^t \, e^{-u B} \Sigma \, \mrm{d} W_u.
$$
Differentiating with respect to $t \in J_{n}$ then yields
$$
\mrm{d} V_t - B V_t \, \mrm{d}t = C V_{t - \tau} \, \mrm{d} t + \Sigma \, \mrm{d} W_t.
$$
This delay SDE has a unique strong solution, which is the centered Gaussian process $Z_{t}$, so that $Z_{t} = V_t$ for all $t$. Thus, the random path $Z= Z([0, T])$ verifies $Z= \Gamma W$. Hence, the probability distributions $\mu$ and $\nu$ of the random paths $W$ and $Z$ on $E$ are related by $\Gamma \mu = \nu$. 

Since $\Gamma : E \to E$ is continuous, the support $S_{\nu}$ of $\nu = \Gamma \mu $ must then be the closure of $\Gamma ( S_{\mu} )$, where $S_{\mu}$ is the support of $\mu$. But for the Brownian path distribution $\mu$, one classically has $S_{\mu} = E$. Thus $\Gamma ( S_{\mu} ) = \Gamma (E)$ contains $\Gamma (\mcal{H}) = \mcal{H}$, and hence its closure $S_{\nu}$ is equal to $E$ since $\mcal{H}$ is dense in $E$. 
If $F$ is any closed vector subspace of $E$ such that $\nu(F)=1$, then $F$ must obviously include the support $S_{\nu} = E$ of $\nu$, and hence $F = E$.

The natural continuous injection $j$ of $E$ into $H = L^2([0, T])$ maps $\nu$ onto a Borel probability $\theta = j(\nu)$. We then have  $\theta = \bar{\Gamma} \mu$, where $\bar{\Gamma} = j \circ \Gamma$ is a continuous linear map from  $E$ into $H$. In the space $H$, the support $S_{\theta}$ of $\theta$ must then be the closure in $H$ of $ K =\bar{\Gamma} (S_{\mu})  $. Hence $K = j (\Gamma (E))$ contains $j(\mcal{H}) = \mcal{H} $. Since $ \mcal{H} $ is dense in $H$, we thus have $S_{\theta} = \oline{K} = H$. This proves the last assertion in Proposition \ref{supportZ}, and concludes the proof.
\end{proof}
\begin{corollary} \label{genericSigma}
The notations are the same as in Proposition \ref{supportZ}. Call $\Sigma$ the matrix of diffusion coefficients for the delay SDE driving the centered Gaussian process $Z_t$. Consider the general case where $\Sigma$ can have any rank $\leqs d$. 
Let $V \subset \mbb{R}^d$ be the vector space generated by the columns of $\Sigma$. Let $\mathcal{V} \subset E= C_{\mbold{0}} ([0, T])$ be the vector space of all paths $f \in E$ which have derivative $f’ \in H = L^2([0, T])$, and verify $f'(t) \in V$ for almost all $t$. 
Call $\nu$ the probability distribution of the random path $Z$ on $E$ as well as on $H$.  

The support of $\nu$ in $E$ is then the closure of $\mathcal{V}$ in $E$. Similarly, the support of $\nu$ in $H$ is equal to the closure of $\mathcal{V}$ in $H$.
\end{corollary}
We skip the detailed proof of this corollary, which we will not use in the large deviations results given further on.
\begin{proposition} \label{injectiveR}
The covariance operator $R : H \to H$ of $Z_{t}$ defined by equation \eqref{covarop} is injective if and only if  the only closed subspace $F$ of $H$ such that $\nu(F) = 1$ is $F= H$, and this property holds if and only if the matrix $\Sigma$ has full rank $d$.
\end{proposition}
\begin{proof}
For $f \in H$, the covariance operator $R$ defines  the quadratic form  
$$
Q(f) = \langle f, R f \rangle_H = \int_0^T \, \int_0^T \, f(s)^*\rho(s,t) f(t) \, \mrm{d}s \, \mrm{d}t.
$$ 
By construction, $Q(f) =\int_0^T \, \int_0^T \, f(s)^* \mbb{E} ( Z_s Z_t^* ) f(t) \, \mrm{d}s \, \mrm{d}t$. By Fubini’s theorem, this yields
\begin{equation*}
Q(f) = \mbb{E} \left[ \int_0^T \, \int_0^T \, ( f(s)^*Z_s ) ( Z_t^* f(t) ) \, \mrm{d}s \, \mrm{d}t \right] = \mbb{E} \left[ \left( \int_0^T \,  f(s)^*Z_s \, \mrm{d}s \right)^2 \right] = 
\mbb{E} [ ( \langle f, Z \rangle_H )^2 ].
\end{equation*}
Let $F(f)$ be the (closed) subspace of all $g \in H$ orthogonal to $f$. The last expression of $Q(f)$ shows that $Q(f) =0$ iff $\langle f, Z \rangle_H =0$ almost surely, and hence iff the random path $Z$ almost surely belongs to $F(f)$, which is equivalent to $\nu(F(f)) = 1$. 
Since the covariance operator $R$ is self-adjoint, one has $R f=0$ iff $Q(f) =0$. Hence, $R f = 0$ is equivalent to $\nu(F(f)) = 1$. So $R$ is injective iff the only closed subspace $F$ of $H$ such that $\nu(F) =1$ is $F= H$. But this property of $\nu$ holds iff the matrix $\Sigma$ is of full rank $d$ thanks to Proposition~\ref{supportZ} and Corollary~\ref{genericSigma}. 
\end{proof}
\begin{theorem} \label{thlambda}
Let $Z_t \in \mbb{R}^d $ be the centered Gaussian process verifying the delay SDE \eqref{e:dSDE-linear-centered} driven by the Brownian motion $W_t \in \mbb{R}^{d}$ with $d \times d$ matrix of diffusion coefficients $\Sigma$. Assume that $\Sigma$ has \emph{full rank} $d$. 
On the Banach space $E = C_{\mbold{0}} ([0,T]) $, the random paths $Z= Z([0, T])$ and $W= W([0, T])$ have respective Gaussian probability distributions $\nu$ and $\mu$. As seen in Proposition \ref{supportZ}, there is a continuous linear map $\Gamma : E \to E$ with continuous inverse $\Gamma^{- 1}$ such that $Z= \Gamma W$. 
On the space $E$, the probabilities $\nu$ and $\mu$ have respective large deviations rate functionals $\lambda_\nu$ and $\lambda_\mu$. 

For any $f \in E$, one then has
\begin{equation} \label{lambda.nu.mu}
\lambda_\nu (f) = \lambda_\mu (\Gamma^{-1} f).
\end{equation}
In particular, $\lambda_{\nu}$ inherits the three well-known properties of $\lambda_\mu$, namely convexity, lower semi-continuity, and inf-compacity. 
Let $\mathcal{H}$ be the subspace of all paths $g$ in $E$ such that $g’$ is in $H = L^2([0, T])$. Then $\lambda_\nu (f)$ is finite if and only if $f$ is in $\mathcal{H}$, and one then has
\begin{equation} \label{energy}
\lambda_{\nu} (f) = \frac{1}{2} \, \int_0^T \; \vnorm{\frac{d}{dt} \Gamma^{-1} f(t)}_{\mbb{R}^{d}}^{2}\,\mrm{d}t = \frac{1}{2} \, \int_0^T \; \vnorm{\Sigma^{-1} [ f’(t) - B f(t) - C f(t - \tau) ]}_{\mbb{R}^{d}}^{2} \, \mrm{d}t 
\end{equation}
for all $f \in \mcal{H}.$  

The vector space $\mathcal{H}$ can be endowed with the Sobolev norm 
\begin{equation} \label{sobolev}
\vnorm{f}_{\mathcal{H}} = \left(\int_0^T \; \vnorm{f’(t)}_{\mbb{R}^{d}}^{2}\, \mrm{d}t \right)^{1/2} ,
\end{equation}
and then becomes a Hilbert space, still denoted $\mathcal{H}$. Convergence in this Hilbert space implies convergence in $E$, but the converse is obviously not true. The restriction of $\lambda_\nu$ to $\mathcal{H}$ defines on the Hilbert space $\mathcal{H}$ a function which is finite, convex, continuous, and has an explicit Gateaux derivative $D \lambda_\nu(f)$ at every $f$ in $\mathcal{H}$.
\end{theorem}

\begin{proof} 
Let $W = W([0, T]) \in E$ be the random path of the Brownian motion $W_t \in \mbb{R}^d$ driving the delay SDE verified by $Z_t$. As proved above in Proposition \ref{supportZ}, there is a bijective continuous linear map $\Gamma: E \to E$ such that $Z= \Gamma W$. The respective probability distributions $\mu$ and $\nu$ of $W$ and $Z$ on the Banach space $E$ are then linked by $\nu= \Gamma \mu$. 
Recall that for any $g \in E$, the rate functional $\lambda_\mu (g)$ of $\mu$ is finite iff $g \in \mathcal{H}$ and is given by  
\begin{equation} \label{energy1}
\lambda_\mu (g) = \frac{1}{2} \int_0^T \, \vnorm{g’(t)}_{\mbb{R}^d}^2 \, \mrm{d}t.
\end{equation}
Since $\nu= \Gamma (\mu)$ with $\Gamma$ continuous, the generic results in Proposition 7.2.6, Subsection 7.2.7, and Proposition 7.2.8 in \cite{Azencott-F-V} show that the rate functional $\lambda_\nu$ of $\nu$ must verify, for all $f \in E$, 
$$
\lambda_\nu (f) = \inf_{ g \in K(f)} \lambda_\mu (g),
$$
where $K(f)$ is the set of all $g \in E$ such that $\Gamma g = f$. Due to Proposition \ref{supportZ}, for each $f \in E,$ the set $K(f)$ contains exactly one function $g = \Gamma^{-1} f$. Hence, we have for all $f \in E$, 
\begin{equation} \label{energy2}
\lambda_\nu (f) = \lambda_\mu (\Gamma^{-1} f).
\end{equation}

As is well known, $\lambda_\mu$ is convex and lower semi-continuous on $E$, and for any fixed $a >0$, the set $\vset{f \in E\, : \, \lambda_\mu(f) \leqs a}$ is compact. Since $\Gamma^{-1}$ is linear and continuous, the same three properties must then also hold for $\lambda_\nu.$ 
As seen in Proposition \ref{supportZ}, the function $\Gamma^{-1} f$ is in $\mathcal{H}$ iff $f \in \mathcal{H}$ and is given by 
\begin{equation} \label{energy3}
\frac{d}{dt} [ \Gamma^{-1} f ] (t) = \Sigma^{-1} [ f’(t) - B f(t) - C f(t - \tau)]  
\end{equation}
for almost all $t \in [0, T]$.
In view of equations \eqref{energy1}, \eqref{energy2}, and \eqref{energy3}, we see that $\lambda_\nu (f)$ is finite iff $f$ is in $\mathcal{H}$, and is given by, for $f\in \mathcal{H},$
\begin{equation} \label{energy4}
\lambda_\nu (f) = \frac{1}{2} \int_0^T \; \vnorm{\Phi f(t)}_{\mbb{R}^{d}}^2 \, \mrm{d}t ,
\end{equation} 
where the linear operator $\Phi : \mathcal{H} \to H$ is defined for $f \in \mathcal{H}$ by 
\begin{equation} \label{energy5}
\Phi f(t)  = \Sigma^{-1} [ f’(t) - B f(t) - C f(t - \tau) ].  
\end{equation} 
On the Hilbert space $\mathcal{H}$, the formulas \eqref{energy4} and \eqref{energy5} show that $\lambda_\nu(f)$ has a Gateaux derivative $D \lambda_\nu (f)$ given by, for all $h \in \mathcal{H}$, 
$$
D \lambda_\nu (f) h = \int_0^T \; \langle \Phi f(t) , \Phi h(t) \rangle_{\mbb{R}^{d}} \, \mrm{d}t.
$$
\end{proof}
\subsection{Optimal paths} \label{optimal}
The centered Gaussian process $Z_{t}$ driven by the delay SDE~\eqref{e:dSDE-linear-centered} does not depend on $\vep$. But the Gaussian diffusion with delay $X_{t} = X_t^{\vep} = m(t) + \vep Z_{t}$ does depend on $\vep.$ 
Let $p$ be any fixed given initial point $p = \gamma (0) = X_0$. Fix the deterministic past initial path $\gamma$ of $X_t$ on $[- \tau, 0]$ and any terminal point $Q \in \mbb{R}^{d}$ as well as a terminal time $T > 0$. 

Since the random path $Z$ and the Brownian path $W$ verify $Z= \Gamma W$, the event $\set{ X^{\vep}_0 = p } \cap \set{ X^{\vep}_T= Q }$ has probability equal to zero for all $\vep > 0$. 
So we fix a very small radius $r >0$, and we now study, as $\vep \to 0$, the probability of events such as
$$
\Omega(p, Q, r) = \{ X^{\vep}_0 = p \} \cap \{ \vnorm{X^{\vep}_T - Q}_{\mbb{R}^{d}} \leqs r \}. 
$$
Let $m(t) = \mbb{E} (X^{\vep}_t)$ and $m$ be the mean path $m([0, T])$. For all $\vep >0$, the centered path $Z$ obviously verifies
\begin{equation} \label{terminal.tube}
\mbb{P} [ X^\vep \in \Omega(p,Q,r) ] = \mbb{P} [ \vep Z \in U(M,r) ],
\end{equation}
where
$$
U(M,r) = \vset{f \in E = C_{\mbold{0}} ([0,T]) \, : \, \vnorm{f(T) - M}_{\mbb{R}^{d}} \leqs r},
$$
and $M = Q - m(T)$.

\begin{proposition} \label{minlambda}
Let $Z$ be the centered Gaussian random path driven by the delay SDE~\eqref{e:dSDE-linear-centered} with full-rank diffusion matrix $\Sigma$. Let $\lambda$ be the large deviations rate functional of $Z \in E= C_{\mbold{0}} ([0,T])$. Let $\rho(s,t) = \mbb{E} (Z_s Z_t^*)$. 
For any point $q \in \mbb{R}^{d},$ call $F_q$ the set of all $f \in E = C_{\mbold{0}} ([0,T])$ such that $f(T) = q$. 

There is then a unique path $f_q \in F_q$ such that
$$
\lambda(f_q) = \Lambda(F_q)  = \inf_{g \in F_q} \lambda(g).
$$
The path $f_q$ starts at $f_q(0)= \mbold{0}$, terminates at $f_q(T) = q$, and is given by 
\begin{equation} \label{e:minimizer}
f_q(s) = \rho(s,T) \rho(T,T)^{-1} q 
\end{equation}
for all $s \in [0, T]$, with rate functional
\begin{equation}
\label{e:Cramer-explicit} 
\lambda(f_q) = \Lambda(F_q) = \frac{1}{2}\langle \rho(T,T)^{-1} q, q \rangle_{\mbb{R}^{d}}.
\end{equation}
In $\mbb{R}^{d}$, fix any open ball $B(M,r)$ of center $M \neq 0$ and small radius $r >0$. The set $U(M,r)$ of all $f \in E $ such that $\vnorm{f(T) - M}_{\mbb{R}^{d}} \leqs r$ then verifies
\begin{equation}\label{minenergy}
\Lambda(U(M,r)) = \inf_{\vset{q \, : \, \vnorm{q-M}_{\mbb{R}^{d}} \leqs r}} \frac{1}{2}\langle \rho(T,T)^{-1} q, q \rangle_{\mbb{R}^{d}}.
\end{equation}

For fixed $M\neq 0$, there are two positive constants $c$ and $C$ such that for $r < c$ one has the following properties:
\begin{enumerate}[leftmargin=*, labelindent=\parindent, itemsep=0.5ex]
\item On the set of paths $U(M,r)$, there is a unique minimizer $f_{M,r}$ of $\lambda.$ \label{Min1}
\item $f_{M,r}$ is the path $f_q$ given by \eqref{e:minimizer} where $q = f_{M,r}(T)$ is the unique point minimizing \eqref{minenergy}. \label{Min2}
\item The minimizing path $f_M$ terminating at $f_M(T) = M$ verifies $\norm{f_{M,r} - f_M}_{E} < C r $. \label{Min3}
\end{enumerate}
Moreover, the closed set $U = U(M,r)$ and its interior $U^\circ$ verify $\Lambda(U) = \Lambda(U^\circ)$.
\end{proposition}
\begin{proof}
Let $H= L^2([0, T])$. The covariance operator $R : H \to H$ of the random path $Z$ is self-adjoint and has finite trace. Since the matrix $\Sigma$ has full rank $d$, we have seen that $R,$  and therefore the self-adjoint operator $S = \sqrt{R},$ are injective.  Hence, $R$ is also positive definite on $H$, with a summable sequence of eigenvalues $\alpha_n > 0$, associated to eigenfunctions $t \to \psi_n(t) \in \mbb{R}^d$. These eigenfunctions $\psi_n$ must be continuous since the integral operator $R$ is defined by the continuous kernel $\rho(s,t)$, which takes values in the set of $d \times d$ matrices and verifies  $\rho(t,s) = \rho(s,t)^*$.
This forces $S$ to also be an integral operator defined on $H$ by a continuous matrix-valued kernel $\sigma(s,t)$, where $\sigma(t,s) = \sigma(s,t)^*$ is a $d \times d$ matrix. This property can be derived from Lemma 3.1 in \cite{FERREIRA}, based on the classical $L^2$-converging expansion 
$$
\sigma (s,t) = \sum_n \alpha_n^{1/2} \psi_n(s) \psi_n(t) ^*.
$$ 
One can show the continuity in $(s,t)$ of $\sigma(s,t)$ by a precise analysis of this series using the decay rates for the $\alpha_n$ provided in \cite{FERREIRA}. In particular, this implies that $g = S h$ is continuous for any $h\in H$.
Since $R$ and $S$ are injective, Proposition \ref{p:General-ld-result} implies that $\lambda(f)$ is finite for $f \in H$ if and only if $f \in S(H)$, and is given in this case by
\begin{equation} \label{Sformula}
\lambda (f) = \frac{1}{2} \int_0^T \; \vnorm{S^{-1} f (t)}_{\mbb{R}^{d}}^2 \, \mrm{d}t.
\end{equation}

In $H,$ consider any open ball $A$ centered at $g = Sh$ for a fixed $h \in H$. We have $\Lambda(A) \leqs \lambda(g) < \infty$ and this implies, due to the large deviations inequalities \eqref{eq:varadhan},
$$
\mbb{P} (\vep Z \in A) > e^{- \vep^{-2} \Lambda(A)} \geqs e^{- \vep^{-2} \lambda(g)} > 0. 
$$
If $g(0)$ were not $\mbold{0}$, we could choose the radius of $A$ small enough to force $A$ and $E = C_{\mbold{0}} ([0,T])$ to be disjoint. Since $\vep Z$ is in $E$ almost surely, we would then have $\mbb{P} (\vep Z \in A) = 0$, which is a contradiction. So we must have $g(0) = \mbold{0}$ and $g = Sh$ must be in $E$. We thus conclude that $S(H)$ is included in $E$. 
Hence $S(H)$ is exactly the set of all $f \in E$ such that $\lambda(f)$ is finite. But due to Theorem \ref{thlambda}, this set coincides also with the set $\mathcal{H}$ of all $f \in E$ having a derivative $f’$ in $H$. Hence we must have $\mathcal{H}= S(H).$ The formula \eqref{Sformula} is hence valid for all $f \in \mathcal{H}.$ 
Let $F_q= \vset{f \in E \, : \, f(T)=q}$. By definition of $\Lambda$, we have 
$$ 
\Lambda(U(M,r)) = \inf_{\vset{q \, : \, \vnorm{q - M}_{\mbb{R}^{d}} \leqs r}} \Lambda(F_q). 
$$
The set of paths $V_q = F_q \cap \mathcal{H}$ is obviously not empty, and $\lambda(f)$ is finite iff $f$ is in $\mathcal{H}$. Hence, we have 
$$
\Lambda(F_q) = \Lambda(V_q) = \inf_{f \in V_q} \lambda (f)  < \infty.
$$

Once endowed with its Sobolev norm, the subspace $\mathcal{H}$ of $E$ becomes a Hilbert space. The linear functional $\mathcal{G} f = f(T) = \int_0^T f’(t) \, \mrm{d}t$ is continuous on this Hilbert space. Then $V_q$ is the closed convex set of all $f \in \mathcal{H}$ such that $\mathcal{G} f = q$. 
To find the minimum $\Lambda(V_q)$ of the Gateaux-differentiable convex functional $\lambda(f)$ on $V_q$, one can hence apply classical Lagrange multiplier theory on the Hilbert space $\mathcal{H}$ (see \cite{Luenberger-1969}). 

On $\mathcal{H}$, the Lagrangian of this minimization under one continuous linear constraint is 
$$
\mcal{L}_{f} = \lambda(f) - \langle \eta , (f(T) - q) \rangle_{\mbb{R}^{d}},
$$
where $\eta \in \mbb{R}^d$ is a vector of Lagrange multipliers and $\langle \cdot , \cdot \rangle_{\mbb{R}^{d}}$ is the scalar product in $\mbb{R}^d$. 
On the Hilbert space $\mathcal{H},$ the Gateaux derivative $D\mcal{L}_{f}$ of $\mcal{L}_{f}$ is given by
\begin{equation} \label{Lagr1}
D\mcal{L}_{f} \varphi = D\lambda(f) \varphi - \langle \eta , \varphi(T) \rangle_{\mbb{R}^{d}} 
\end{equation}
for all $\varphi \in \mathcal{H}$,
where $D\lambda(f)$ is the Gateaux derivative of $\lambda$ on the Hilbert space $\mathcal{H}$. 
For $f \in\mathcal{H}$, the Gateaux derivative $D\lambda(f)$ of the quadratic form $\lambda(f)$ defined by \eqref{Sformula} verifies for all $\varphi \in \mathcal{H}$
$$
D\lambda(f) \varphi = \int_0^T \; \langle S^{-1}f (t) , S^{-1}\varphi(t) \rangle_{\mbb{R}^{d}} \, \mrm{d}t.
$$
Fix any $f \in V_q$ minimizing $\lambda(f)$ on $V_q$. Then, $f \in \mathcal{H}$ and verifies $D\mcal{L} (f) = 0$. Fix an arbitrary $h \in H$. The function $\varphi = S h$ is then in $ S(H) = \mathcal{H}$. In view of equation \eqref{Lagr1}, we then have
\begin{equation} \label{Lagr2}
0 =D\mcal{L}_{f} \varphi = 
\int_0^T \; \langle S^{-1}f (t) , h(t) \rangle_{\mbb{R}^{d}} \, \mrm{d}t -  \langle \eta, Sh(T) \rangle_{\mbb{R}^{d}}.
\end{equation}
Using the continuous matrix-valued kernel $\sigma(s,t)$ of $S$, we have 
$$
\langle \eta, Sh(T) \rangle_{\mbb{R}^{d}} = \int_0^T \;  \langle\eta, \sigma(T,t)h(t) \rangle_{\mbb{R}^{d}}\, \mrm{d}t =
\int_0^T \; \langle \sigma(T,t)^* \eta, h(t) \rangle_{\mbb{R}^{d}} \, \mrm{d}t = \int_0^T \; \langle \sigma(t, T) \eta, h(t) \rangle_{\mbb{R}^{d}} \, \mrm{d}t.
$$
Hence, equation \eqref{Lagr2} becomes
$$
0 = \int_0^T \; \langle S^{-1}f (t) - \sigma(t, T) \eta, h(t) \rangle_{\mbb{R}^{d}} \, \mrm{d}t. 
$$
Since $h$ is arbitrary in $H= L^2([0, T])$, we conclude that $S^{-1}f (t) = \sigma(t, T) \eta$ for almost all $t \in [0, T]$.
Hence, we have for all $s \in [0, T]$,
$$
f(s) = S [S^{-1}f] (s) = \int_0^T \; \sigma(s,t) S^{-1}f(t) \, \mrm{d}t = 
\left[\int_0^T \; \sigma(s,t) \sigma(t, T) \, \mrm{d}t \right] \eta.
$$
Since $S^2 = R$, the last integral is obviously equal to $\rho(s,T)$, and we obtain 
$$
f(s) = \rho(s,T) \eta 
$$
for all $s \in [0,T]$.
Since $f$ is in $V_q,$ we have $f(T) = q$ so that $\rho(T,T) \eta= q$. The unique minimizer $f_q$ of $\lambda(f)$ in $V_q$ is then given by 
\begin{equation} \label{q0}
f_q(s) = \rho(s,T) \rho(T,T)^{-1} q   
\end{equation}
for all $s \in [0,T]$.

To compute $\Lambda(F_q) = \Lambda(V_q) =\lambda(f_{q})$, we write
$$
\lambda(f_{q}) = \frac{1}{2} \int_0^T \;  \langle S^{-1} f_{q}(t), S^{-1} f_{q}(t) \rangle_{\mbb{R}^{d}} \, \mrm{d}t = 
\frac{1}{2} \int_0^T \; \langle \sigma(t, T) \eta, \sigma(t, T) \eta \rangle_{\mbb{R}^{d}} \, \mrm{d}t.
$$
Since $\sigma(t, T)^* = \sigma(T, t)$, the last integral becomes 
$$
\frac{1}{2} \left\langle \eta, \left[ \int_0^T \; \sigma(T, t) \sigma(t, T)\,  \mrm{d}t \right] \eta \right\rangle_{\mbb{R}^{d}} = \frac{1}{2} \langle \eta, \rho(T,T) \eta\rangle_{\mbb{R}^{d}} = \frac{1}{2}\langle \rho(T,T)^{-1} q, q \rangle_{\mbb{R}^{d}}.
$$
We have thus proved that 
$$
\Lambda(F_q) = \lambda(f_q) =  \frac{1}{2}\langle \rho(T,T)^{-1} q, q \rangle_{\mbb{R}^{d}},
$$
which implies 
\begin{equation} \label{q1}
\Lambda(U(M,r)) = \inf _{\vset{q \, : \, \vnorm{q - M}_{\mbb{R}^{d}} \leqs r}} \lambda(f_q) = \min_{\vset{q \, : \, \vnorm{q - M}_{\mbb{R}^{d}} \leqs r}} \frac{1}{2}\langle \rho(T,T)^{-1} q, q \rangle_{\mbb{R}^{d}}.
\end{equation}

A linear transformation by $\frac{1}{\sqrt{2}}\sqrt{\rho(T,T)^{-1} }$ replaces this last minimization by finding the closest point to $\mbold{0}$ within a small hyper-ellipsoid $J$ centered at $M$. 
Impose $r <c$ with $c = \frac{\vnorm{M}_{\mbb{R}^{d}}}{\vnorm{\rho(T,T)^{-1}}}$ to be sure that $\mbold{0}$ is not in $J$. Since $J$ is strictly convex, there is a unique point $x \in J$ minimizing the convex function $\vnorm{x}_{\mbb{R}^{d}}^2$, and $x$ lies on the boundary of $J$.
So for $r <c $ and $\vnorm{q - M}_{\mbb{R}^{d}} \leqs r,$ there is a unique $q = q(M,r)$ minimizing $\lambda(f_q)$ and $\vnorm{q-M}_{\mbb{R}^{d}} = r$. Due to formula \eqref{q1}, the corresponding extremal $f_q$, which we denote $f_{M,r},$ is then the unique path minimizing $\lambda$ on $U(M,r)$. This proves the assertions (\ref{Min1}) and (\ref{Min2}). 
By formula \eqref{q0} with $q= q(M,r)$ and $t \in [0, T]$, we have
$$
\vnorm{f_{M,r}(t) - f_M(t)}_{\mbb{R}^{d}}\leqs \vnorm{\rho(t,T)}\vnorm{\rho(T,T)^{-1}} \vnorm{q - M}_{\mbb{R}^{d}} \leqs C r,
$$
where 
$C = \vnorm{\rho(T,T)^{-1}} \max_{t \in [0,T]} \vnorm{\rho(t,T)}$.
This proves assertion (\ref{Min3}).

The interior $U^{\circ}$ of $U= U(M,r)$ is obviously the set of all $f \in E$ such that $\vnorm{f(T) - M}_{\mbb{R}^{d}} <r$. One can then replicate the preceding proof replacing $U$ by $U^\circ$ and the inequalities $\vnorm{q - M}_{\mbb{R}^{d}} \leqs r $ by $\vnorm{q - M}_{\mbb{R}^{d}} < r $ to obtain 
$$
\Lambda(U^{\circ}) = \inf _{\vset{q \, : \, \vnorm{q - M}_{\mbb{R}^{d}} < r}} \lambda(f_q) = \inf _{\vset{q \, : \, \vnorm{q - M}_{\mbb{R}^{d}} < r}} \frac{1}{2}\langle \rho(T,T)^{-1}q, q \rangle_{\mbb{R}^{d}} = \Lambda(U).
$$
\end{proof}
\subsection{Probabilistic interpretation}
\label{ss:tube}
For any $f \in E$ and small $r >0$ call $\Tube(f,r)$ the open ball of center $f$ and radius $r$ in $E$, which can be viewed as a thin tube around the path $f$.
\begin{proposition} \label{probtube1}
Let $Z \in E= C_{\mbold{0}} ([0,T])$ be the centered Gaussian random path driven by the delay SDE~\eqref{e:dSDE-linear-centered} with matrix $\Sigma$ of full rank $d$.  Call $\lambda$ the large deviations rate functional of $Z$ on $E$. Let $G= \Tube(f,r) \subset E$ be any open tube with center $f \in E$ and radius $r$.
We then have $\Lambda(G) = \Lambda(\bar{G})$, as well as  the following limit:
\begin{equation} \label{limprobtube}
\lim_{\vep \to 0} \; \frac {1}{\vep^2} \log \mbb{P} [ \vep Z \in G ] =  - \Lambda(G) = - \Lambda(\bar{G}).
\end{equation}
\end{proposition}
\begin{proof}
Since smooth paths are dense in $E$, the open tube G contains at least one smooth path $u$. Then $u$ is in $\mathcal{H}$, which forces $\lambda(u)$ to be finite due to Theorem \ref{thlambda}. In particular, both $\Lambda(G)$ and $\Lambda(\bar{G})$ are finite. 
We have seen that $\lambda(g)$ is lower semi-continuous for $g \in E$ and has the inf-compactness property. Hence, $\lambda$ reaches its minimum on any closed subset of $E$. So, $\bar{G}$ must contain a path $f$ such that $\lambda(f) = \Lambda(\bar{G})$. Thus,  $\lambda(f)$ is finite, which implies $f \in \mathcal{H}.$ 
Since $f \in\bar{G}$ and $u \in G$, the convex open tube $G$ will contain $g_n = (1-1/n) f + (1/n) u$ provided $n$ is large enough. However, the sequence $g_n$ is in $\mathcal{H}$ and as $n \to \infty$ clearly converges to $f$ in the Hilbert space $\mathcal{H}$. As proved above, $\lambda$ is continuous on this Hilbert space, so that $\lim_{n \to \infty} \lambda(g_n) = \lambda(f)$. Since $g_n \in G$, we have $\Lambda(G) \leqs \lambda(g_n)$. This yields as $n \to \infty$ 
$$
\Lambda(G) \leqs \lambda(f) = \Lambda(\bar{G}).
$$
Since $\Lambda$ is a decreasing set functional, one always has $\Lambda(\bar{G}) \leqs \Lambda(G) $, so we conclude that $\Lambda(G) = \Lambda(\bar{G}).$
Equation \eqref{limprobtube} is then a direct consequence of the generic large deviations inequalities (see equation \eqref{eq:varadhan}).
\end{proof}
\begin{proposition} \label{probtube2}
In $\mbb{R}^d$, fix any terminal point $M \neq \mbold{0}$. Let $U(M,r)$ be the set of all paths $g \in E = C_{\mbold{0}} ([0,T])$ such that $\vnorm{g(T) - M}_{\mbb{R}^{d}} \leqs r$. As proved above, for any  $r$ small enough, there is a unique path $f_{M,r}$ in $U(M,r)$ which minimizes $\lambda$ on $U(M,r)$.
In the space $E$, let $G$ be the open tube of center $f_{M,r}$ and radius $r$. We then have
$$
\Lambda(G) = \Lambda(\bar{G}) = \lambda(f_{M,r}). 
$$

Fix $M \neq \mbold{0}$ and $r$ small enough. One has then the precise large deviation result
\begin{equation} \label{limprobtube2}
\lim_{\vep \to 0} \; \frac {1}{\vep^2} \log \mbb{P} [ \vep Z \in \Tube(f_{M, r}, r) ]= - \lambda(f_{M,r}). 
\end{equation}
\end{proposition}
\begin{proof}
Fix a terminal point $M \neq \mbold{0}$. Let $U = U(M,r)$ be the set of all $g \in E$ such that $\vnorm{g(T) - M}_{\mbb{R}^{d}} \leqs r$. By Proposition \ref{minlambda}, for $r$ small enough, there is a unique $f = f_{M,r}$ in $U$ such that $\Lambda(U) = \lambda(f)$.
The open tube $G = \Tube(f, r)$ is included in the closed set $U$. On the space $E$, the function $\lambda$ is lower semi-continuous and has the inf-compactness property, and hence must reach its minimum on any closed subset of $E$. Since $U - G$ is closed in $E$, there must then be a $g \in U - G$ such that $\lambda(g) = \Lambda(U-G)$. We have then 
$$
\lambda(g) = \Lambda(U - G) \geqs \Lambda(U) = \lambda(f).
$$
If $\lambda(g) = \lambda(f)$, we would have $g = f$ because $g \in U$ and $f$ is the unique minimizer of $\lambda$ in $U$. But we cannot have $g = f$ because $g$ is not in $G$ and $f$ is in $G$. We thus conclude that $\lambda(g) > \lambda(f)$, and hence $\Lambda(U - G)  > \Lambda(U)$. 
By definition, $\Lambda(A)$ is the infimum of $\lambda$ on $A$, so that 
$$
\Lambda(U) = \min \vset{\Lambda(U - G), \Lambda(G)}. 
$$
Since $\Lambda(U - G)  > \Lambda(U)$, we deduce that 
$\Lambda(G) = \Lambda(U) = \lambda(f)$.
Applying Proposition \ref{probtube1} to the tubes $G$ and $\bar{G}$ then concludes the proof. 
\end{proof}

\subsection*{Probabilistic interpretation for the process $X^{\vep}$}

The preceding results can immediately be interpreted in terms of the non-centered random path $X^{\vep} = m + \vep Z$. 
Fix the past initial path $\gamma$ of the Gaussian diffusion with delay $X_t = X^\vep_t$. Set $p= \gamma(0)$. Fix any terminal point $Q \neq m(T)$. Let $M= Q- m(T)$. Fix $r$ small enough. Let $q = q(T, m(T), r)$ be the unique point minimizing the quadratic form $\frac{1}{2}\langle \rho (T,T)^{-1} (q - m(T)), (q - m(T))\rangle_{\mbb{R}^{d}}$ on the sphere of center $m(T)$ and radius $r$.
When $\vep \to 0$, the most likey path $g^T$ realizing the rare event 
$$
\Omega(p,Q,r) = \vset{ X^{\vep}_0 =p } \cap \vset{ \vnorm{X^{\vep}_{T} - Q}_{\mbb{R}^{d}} \leqs r }
$$ 
is then given by
$$
g^{T} (s) = m(s) + \rho (s,T) [\rho (T,T)^{-1} (q-m(T))] \qquad (0 \leqs s \leqs T).
$$
Call $\lambda_X$ the large deviations rate functional for the process $X^{\vep}_t$. One has then 
$$
\lambda_X (g^{T}) = \frac{1}{2} [\rho (T,T)^{-1} (q - m(T))] \cdot [q - m(T)],
$$
where $\cdot$ is the standard dot product in $\mbb{R}^{d}.$  

In the Banach space of continuous paths $C([0, T]),$ consider any open tube $\Tube(g^T, r)$ of center $g^T$ and small radius $r$. One has then the large deviations limit
\begin{equation} \label{limprobtube3}
\lim_{\vep \to 0} \; \frac {1}{\vep^2} \log \mbb{P} [ X^\vep \in \Tube(g^T, r)] = - \lambda_X (g^T). 
\end{equation}

With the notations introduced earlier, the most likely transition path $g^T$ is of the form $g^T(t) = m(t) + f_{M,r},$ where $M= Q- m(T)$. As proved earlier, as $r \to 0$, the minimizing path $f_{M,r}$ converges to $f_M$ in $E = C_{\mbold{0}} ([0,T])$ at linear speed $C r$. Recall that $f_M$ is the fully explicit path minimizing $\lambda$ over all the paths $f \in E$ such that $f(T) = M$. Hence, as $r \to 0$, the path $g^T$ will converge in uniform norm to $h^T$ given by 
\begin{equation}
\label{e:Xt-minimizer}
h^{T} (s) = m(s) + \rho (s,T) [\rho (T,T)^{-1} (Q-m(T))] \qquad (0 \leqs s \leqs T),
\end{equation}
and as $r \to 0$, the energy $\lambda_X (g^T)$ converges to
\begin{equation}
\label{e:Xt-minimizer-energy}
\lambda_X (h^{T}) = \frac{1}{2} [\rho (T,T)^{-1} (Q - m(T))] \cdot [Q - m(T)].
\end{equation}

Minimizing $\lambda_X (h^{T})$ over $T$ will hence produce the most likely time $\hat{T}$ of transition from $X^\vep_0 = p$ to a very small neighborhood of $Q$ at time $\hat{T}$, and the most likely transition path will be very close to $h^{\hat{T}}$.

Note that $Q = m(T)$ is a very different situation because $\mbb{P} [\Omega(p, m(T), r)]$ tends to $1$ as $\vep \to 0$. In that case, $X^{\vep}_{t}$ will remain within a fixed small tube around the mean path $m(t)$ with probability converging to $1$ as $\vep \to 0$.

\section{Numerical implementation}
\label{s:implementation}
For numerical computations, we naturally consider only the limit case $r=0$. We have implemented a numerical scheme in three steps:
\begin{itemize}
\item  Solve several delay ODEs to compute the mean path $m(t)$ of $X_{t}$ and the covariance function $\rho(s,t)$ of $Z_{t}$.
\item  For fixed $T, p, q,$ compute the most likely transition path $h^T$ from $X_0 =p$ to  $X_T =q$, and its energy  $\lambda_X(h^T)$, as given by~\eqref{e:Xt-minimizer} and \eqref{e:Xt-minimizer-energy}.
\item Compute the optimal transition time $T_{\text{opt}}$ by minimizing $\lambda_{X} (h^T)$ over all times $T>0.$
\end{itemize}

\begin{notation}
From this point forward, we write $\lambda$ for the rate functional associated with the process $X_{t}^{\vep}$.
\end{notation}

\subsection{Numerical solution of three delay ODEs }
Each delay ODE of interest here is iteratively solved on the time intervals $J_k = [(k-1)\tau, k\tau]$ for $k = 1,2, \ldots, (1 + \lfloor T/\tau \rfloor)$. For each $k$, this amounts to numerically solving a linear ODE with known right-hand side. For this, we use a backward Euler scheme, which is known to be stable for equations of this form~\cite{Bellen-Zennaro-2003, Guglielmi-1998}. 
To compute $m(t)$, we discretize $[0,T]$ into subintervals of equal length $\Delta t =\tau / N$. Backward Euler is given by
\begin{equation*}
m(t) - m(t-\Delta t) = [a + B m(t) + C m(t-\tau)]\Delta t,
\end{equation*}
which yields the recursive equation
\begin{equation*}
m(t) = (I - \Delta t B)^{-1}m(t-\Delta t) + \Delta t(I - \Delta tB)^{-1}[a + C m(t-\tau)].
\end{equation*}
The initial history of the mean is used to numerically compute  the solution $m(t) = m_1(t)$ on the initial interval $J_1$. To numerically generate the solution  $m(t) = m_k(t)$ on $J_k$, we proceed by iteration on $k$, using the discretized backward Euler expressions.
This yields a full numerical approximation of $m(t)$ on $[0,T].$ 

We apply a completely similar strategy to compute for each $s$ the function $t \mapsto \phi_s(t)$ defined by \eqref{eq:phi}. 
Here, both $s$ and $t$ will be constrained to belong to the finite grid 
\begin{equation*}
\Grid (N) =  \vset{ j \tau / N : j=1, \ldots, M \text{ and } M = N (1 + \lfloor T/\tau \rfloor ) }.
\end{equation*}
After the computation of $ \phi_s(t),$ we generate the $F(s,t)$ values for $s,t \in \Grid (N)$ by the explicit formula $F(s,t) = \phi_s(t) - \Sigma^{*} H(s-t)$, where $H(s-t)$ is a Heaviside function.

We then proceed to compute $\rho(s,t)$ for $s,t \in \Grid (N)$. For each fixed $t \in \Grid (N)$, the backward Euler discretization of the delay ODE verified by the function $s \mapsto \rho(s,t)$ yields the recursive relation
\begin{equation}
\label{e:rho-numerics}
\rho(s,t) = (I - \Delta s B)^{-1}\rho(s - \Delta s, t) + \Delta s(I - \Delta s B)^{-1}[C\rho(s - \tau, t) + \Sigma F(s,t)],
\end{equation}
where $\Delta s = \tau/N .$ The initial values $\rho(s,t) = 0$ for $s\in [-\tau, 0]$ and the recursive relation ~\eqref{e:rho-numerics} enable the computation of $\rho(s,t)$ for $s\in J_1.$ Keeping $t$ fixed, one then uses~\eqref{e:rho-numerics} and the values of $\rho(\cdot,t)$ on $J_k$ to compute the values of  $\rho( \cdot ,t)$ on $ J_{k+1}$. 
Repeating this operation for each $t \in \Grid (N)$ finally provides $\rho(s,t)$ for all $s,t \in \Grid (N)$.
\subsection{Numerical minimization of the Cramer transform}

Fix $T > 0$. For the Gaussian diffusion with delay  $X_{t}$, the most likely  transition path $h^{T}$ from $X_{0} = p$ to $X_{T} = q$ and its energy $\lambda (h^{T})$ have been explicitly expressed in terms of the functions $m(t)$ and $\rho(s,t)$ (see~\eqref{e:Xt-minimizer} and~\eqref{e:Xt-minimizer-energy}).
Plugging the values of $m(t)$ and $\rho(s,t)$ numerically computed for $s,t \in \Grid (N)$ into \eqref{e:Xt-minimizer} and \eqref{e:Xt-minimizer-energy} thus provides numerical approximations of $\lambda (h^{T})$ and $h^{T} (s)$ for $s \in \Grid (N)$.

To compute the most likely time at which $X_t$ will reach $q$, whenever   this rare event is realized, we have to minimize $\lambda (h^{T})$ in $T$. So we select a large terminal time $T_{\text{large}}$, and we numerically minimize the function $\lambda (h^{T})$ on the interval $[0,T_{\text{large}}].$ If on that time interval $\lambda (h^{T})$ exhibits an actual minimum at $T_{\text{opt}}$, this gives us the (numerically approximate) most likely transition time $T_{\text{opt}}$. Otherwise, we set $T_{\text{opt}} = \infty.$
\subsection{Exit path from metastable stationary states}
As $\vep \to 0$, the dynamics of $X_t$ limit to a deterministic dynamical system $x_t$ driven by a first-order delay ODE. Let $p$ be a stable stationary state of $x_t$, and let $V$ be a small neighborhood of $p$.  Determining, for small $\vep$, the most likely  path followed by $X_t$ to exit from $V$ when $X_0 = p$ is a problem of practical interest in many contexts. 
Our numerical computation of the most likely transition path from $X_{0} = p$ to $X_{T} = q$  with $q$ on the boundary of $V$ will enable us  to numerically solve these types of    exit problems.
We now illustrate this approach with the detailed study of a specific dynamical system from biochemistry.
\section{Linear noise approximations for excitable systems}
\label{s:LNAs}
\subsection{Excitable systems from biochemistry}
We begin by explaining the importance of noise, delay, and metastability for the dynamics of genetic regulatory circuits. Such circuits may be described by delay SDEs~\cite{Brett-Galla-PRL-2013, Gupta_jcp2014} and represent a significant class of systems to which our approach can be applied.

Cellular noise and transcriptional delay shape the dynamics of genetic regulatory circuits.
Stochasticity within cellular processes arises from a variety of sources.
Sequences of chemical reactions at low molecule numbers produce an intrinsic form of noise.
Multiple other types of variability affect dynamics across spatial and temporal scales.
Examples include fluctuations in environmental conditions, metabolic processes, energy availability, et cetera.
Noise functions constructively in both microbial and eukaryotic cells and on multiple timescales.
It enables probabilistic differentiation strategies for cell populations, such as stochastic state-switching in bistable circuits and transient cellular differentiation in excitable circuits (\textit{e.g.}~\cite{Eldar-Elowitz-Functional-2010, Suel_2006, Davidson-Surette-Individuality-2008}).

Certain circuit architectures such as toggle switches and excitable circuits enable noise-induced rare events.
These architectures allow cellular populations to probabilistically switch states in response to environmental fluctuations~\cite{Eldar-Elowitz-Functional-2010}.

Bistability is a central characteristic of biological switches.
It is essential in the determination of cell fate in multicellular organisms~\cite{hong:2012}, the regulation of cell cycle oscillations during mitosis~\cite{he:2011}, and the maintenance of epigenetic traits in microbes~\cite{ozbudak:2004}.
Metastable states can be created by positive feedback loops.
Once a trajectory enters a metastable state, it will typically remain there for a considerable amount of time before noise induces a transition~\cite{kepler:2001, Eldar-Elowitz-Functional-2010}.
This phenomenon has been studied in many contexts, including the lysis/lysogeny switch of bacteriophage $\lambda $~\cite{aurell:2002, warren:2005}, bacterial persistence~\cite{balaban:2004}, and synthetically constructed positive feedback loops~\cite{gardner:2000, nevozhay:2012}.

Many biological systems exhibit excitability~\cite{Meeks_2002,Dupin29042003,Suel_2006}.
Excitable systems commonly feature a single metastable state bordered by a sizable, active region of phase space.
When stochastic fluctuations cause a trajectory to exit the basin of attraction of this metastable state, the trajectory will make a large excursion before returning to the basin.
Transient differentiation into a genetically competent state in \textit{Bacillus subtilis}, for example, is enabled by an excitable circuit architecture.
Positive feedback controls the threshold for competent event initiation, while a slower negative feedback loop controls the duration of competence events~\cite{Suel_2006, Maamar-2005, Maamar-2007, Suel-2007, Cagatay-2009}.
Rare events in such excitable systems manifest as bursts of activity.

\subsection{General linear noise approximations (LNAs)}
\label{ss:General-LNAs}

We explain how Gaussian diffusions driven by delay SDEs such as~\eqref{e:dSDE-linear} arise from linear noise approximations of nonlinear delay SDEs.
Brett and Galla~\cite{Brett-Galla-PRL-2013} introduced linear noise approximations for chemical Langevin equations modeling biochemical reaction networks.
Consider the delay SDE
\begin{equation}
\label{e:dSDE-nonlinear}
\mrm{d} x_{t} = f(x(t),x(t - \tau )) \, \mrm{d}t + \frac{1}{\sqrt{N }} g(x(t),x(t - \tau )) \, \mrm{d} W_{t}.
\end{equation}
Here $f : \mbb{R}^{d} \times \mbb{R}^{d} \to \mbb{R}^{d}$, $g : \mbb{R}^{d} \times \mbb{R}^{d} \to \mbb{R}^{d \times n}$, $W_{t}$ denotes standard $n$-dimensional Brownian motion, and $N > 0$ denotes system size (characteristic number of molecules in a biochemical system).
Notice that we allow both the drift and the diffusion to depend on the past.
Suppose $x^{\infty } (t)$ solves the deterministic limit of~\eqref{e:dSDE-nonlinear}; that is, $x^{\infty } (t)$ solves
\begin{equation}
\label{e:elway}
\mrm{d} x_{t} = f(x(t),x(t - \tau )) \, \mrm{d}t.
\end{equation}
As we have indicated in our introduction, around a stable point $z$ of the limit ODE as $N$ tends to infinity, one can approximate such a system by a Gaussian diffusion with delay and small diffusion matrix $\frac{1}{\sqrt{N}} \Sigma$.
Define $\xi (t)$ by
\begin{equation*}
x(t) = x^{\infty } (t) + \xi (t).
\end{equation*}
Substituting this ansatz into~\eqref{e:dSDE-nonlinear} and performing Taylor expansions of $f$ and $g$ based at the deterministic trajectory yields the linear noise approximation
\begin{equation}
\label{e:lna-general}
\begin{split}
\mrm{d} \xi_{t} &= \left[ D_{1} f (x^{\infty } (t), x^{\infty } (t - \tau )) \xi (t) + D_{2} f (x^{\infty } (t), x^{\infty } (t - \tau )) \xi (t - \tau ) \right] \mrm{d} t
\\
&\quad {}+ \frac{1}{\sqrt{N}} g (x^{\infty } (t), x^{\infty } (t - \tau )) \, \mrm{d} W_{t}.
\end{split}
\end{equation}
Here $D_{1}$ and $D_{2}$ denote differentiation with respect to the first and second sets of $d$ arguments, respectively.
If $x^{\infty } (t)$ happens to be a stable fixed point of~\eqref{e:elway}, say $x^{\infty } (t) \equiv z$, then~\eqref{e:lna-general} becomes
\begin{equation*}
\mrm{d} \xi_{t} = \left[ D_{1} f (z, z) \xi (t) + D_{2} f (z, z) \xi (t - \tau ) \right] \mrm{d} t + \frac{1}{\sqrt{N}} g(z, z) \, \mrm{d} W_{t}.
\end{equation*}
This is~\eqref{e:dSDE-linear} with $a = \mbold{0}$, $ B= D_{1} f (z, z)$, $C = D_{2} f (z, z)$, $\Sigma = g(z,z)$, and $\varepsilon = \frac{1}{\sqrt{N}}.$

\section{A bistable biochemical system}
\label{s:toggle}

\subsection{Chemical Langevin equation}

The genetic toggle switch we study consists of two protein species, each of which represses the production of the other.
We model the switch using the chemical Langevin equation
\begin{subequations}
\label{e:dCLE}
\begin{align}
\mrm{d} x &= \left( \frac{\beta }{1 + y(t - \tau )^{2} / k} - \gamma x \right) \mrm{d}t + \frac{1}{\sqrt{N }} \left( \frac{\beta }{1 + y(t - \tau )^{2} / k} + \gamma x \right)^{\tfrac{1}{2}} \mrm{d} W_{1}
\\
\mrm{d} y &= \left( \frac{\beta }{1 + x(t - \tau )^{2} / k} - \gamma y \right) \mrm{d}t + \frac{1}{\sqrt{N }} \left( \frac{\beta }{1 + x(t - \tau )^{2} / k} + \gamma y \right)^{\tfrac{1}{2}} \mrm{d} W_{2},
\end{align}
\end{subequations}
where $x$ and $y$ denote the concentrations of the two protein species, $\beta $ denotes maximal protein production rate, $k$ is the protein level at which production is cut in half, $\gamma $ is the dilution rate, $N$ denotes system size, and $W_{1}$ and $W_{2}$ are independent standard Brownian motions.
Notice that~\eqref{e:dCLE} is a symmetric system.
In the deterministic limit as $N \to \infty$, the co-repressive toggle switch is described by the reaction rate equations
\begin{subequations}
\label{e:dRRE}
\begin{align}
\mrm{d} x &= \left( \frac{\beta }{1 + y(t - \tau )^{2} / k} - \gamma x \right) \mrm{d}t
\\
\mrm{d} y &= \left( \frac{\beta }{1 + x(t - \tau )^{2} / k} - \gamma y \right) \mrm{d}t.
\end{align}
\end{subequations}
System~\eqref{e:dRRE} has two stable stationary states, $(x_{\text{low}}, y_{\text{high}})$ and $(x_{\text{high}}, y_{\text{low}})$, as well as a saddle stationary state $(x_{s}, y_{s})$.
See~\cite[Figure~7]{Gupta_jcp2014} or~\cite[Figure~3A, inset]{Cuba-cell-cycle-noise-2016} for a phase portrait of~\eqref{e:dRRE}. 

In the stochastic ($N < \infty $) regime, a typical trajectory of the co-repressive toggle switch will spend most of its time near the metastable states, occasionally hopping from one to the other~\cite[Figure~3A]{Cuba-cell-cycle-noise-2016}.
Such rare events raise interesting questions.
For large $N$, is the co-repressive toggle switch well-approximated by a two-state Markov chain on long timescales?
If so, what are the transition rates?
To determine these rates, one would need to compute both a quasipotential and a formula of Eyring-Kramers type.

Here, we focus on the problem of optimal escape from neighborhoods of metastable states.
We fix a neighborhood of $(x_{\text{low}}, y_{\text{high}})$ (Figure~\ref{fig:ExampleExits}, black curve) and ask:
What is the most likely route of escape from this neighborhood for~\eqref{e:dCLE}?
In Section~\ref{ss:toggle-LNA}, we compute a linear noise approximation of~\eqref{e:dCLE} that is valid near $(x_{\text{low}}, y_{\text{high}})$.
Since this linear noise approximation is a Gaussian diffusion with delay, the framework of the present paper applies to it.
We use this framework to compute most likely routes of escape for the linear noise approximation and thereby obtain (approximate) most likely routes of escape for~\eqref{e:dCLE}.

\begin{figure}[ht]
\begin{center}
\includegraphics[scale = 0.5]{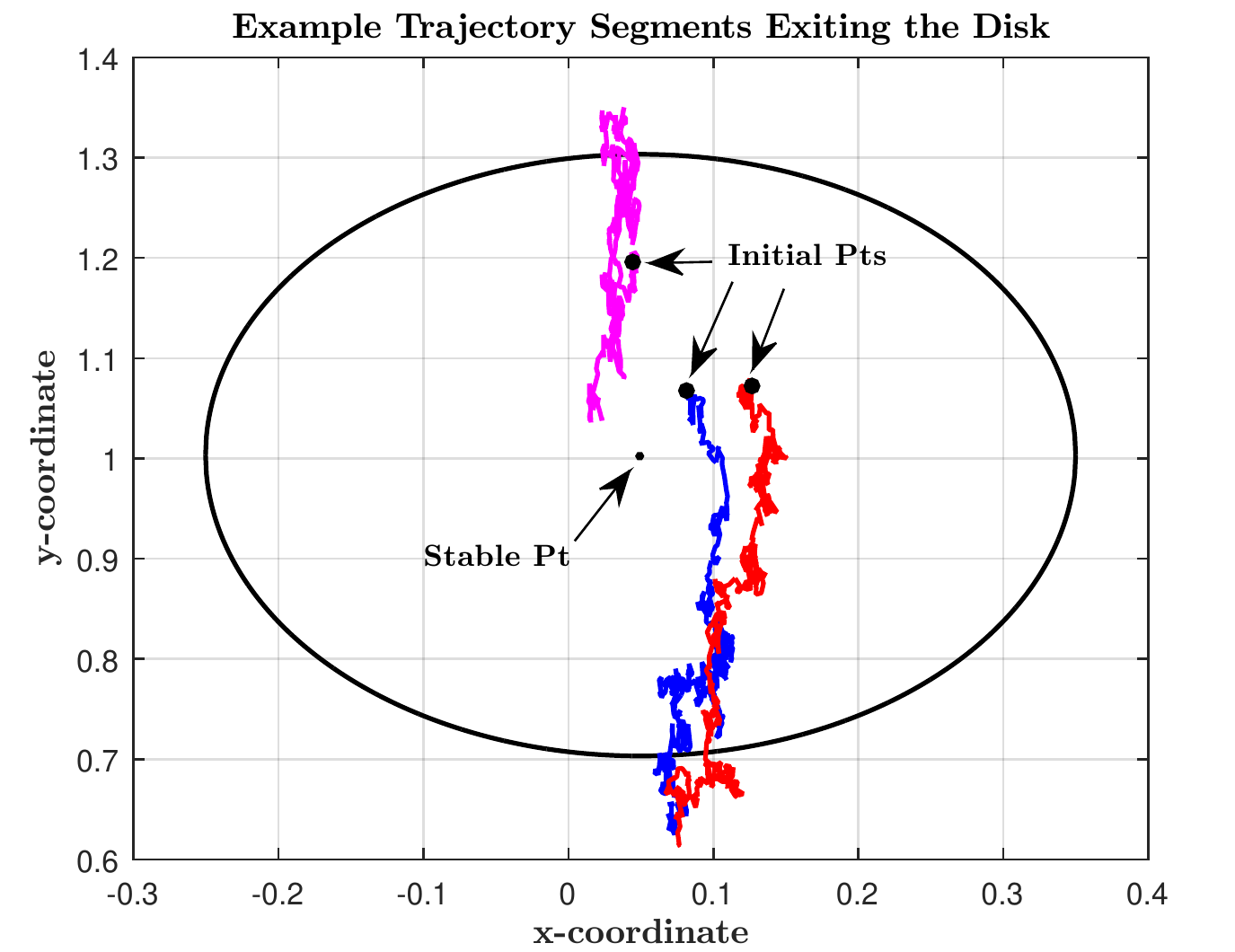}
\caption{Sample trajectory segments of~\eqref{e:dCLE} in a neighborhood of the metastable state $(x_{\text{low}}, y_{\text{high}})$. We simulated $1000$ trajectories over the time interval $[0,5]$. We then chose three sample trajectories that exited the disk $D$ and extracted a segment from each of them. The blue, red, and magenta trajectory segments begin near the metastable state (small black disk) at the coordinates $(0.0817, 1.0668)$, $(0.0673, 1.1233)$, and $(0.1272, 1.0733)$, respectively, and cover time intervals $[2.799, 3.399]$, $[3.099, 3.599]$, and $[1.699, 2.299]$. The history of each simulated trajectory over the time interval $[0,5]$ is taken to be fixed at $(x_{\text{low}}, y_{\text{high}})$ over the time interval $[-\tau , 0]$.
Trajectories have been generated using Euler-Maruyama with time step $\Delta t = \tau / 1000 = 0.001$.
Parameters: $\beta = 0.73$, $k = 0.05$, $\gamma = \ln (2)$, $\tau = 1$, $N=30$.}
\label{fig:ExampleExits}
\end{center}
\end{figure}

\subsection{Approximation by Gaussian Diffusions with delays}
\label{ss:toggle-LNA}

We study an approximation of~\eqref{e:dCLE} by Gaussian diffusions with delay that is valid in a neighborhood of $(x_{\text{low}}, y_{\text{high}}) \defasr (v,w)$.
Writing
\begin{equation*}
x(t) = v + \xi_{1} (t), \qquad y(t) = w + \xi_{2} (t),
\end{equation*}
the Gaussian diffusion with delay is given by
\begin{equation}
\label{e:LNA}
\begin{aligned}
\mrm{d} \xi_{1}(t) &= \left(-\gamma \xi_{1}(t) - \frac{2 \beta w}{k[1 + w^2/k]^2} \xi_{2}(t - \tau)\right)\mrm{d} t + \frac{1}{\sqrt{N}} \left(\frac{\beta}{1 + w^2/k} + \gamma v\right)^{1/2} \mrm{d} W_{1}(t),\\
\mrm{d} \xi_{2}(t) &= \left(-\gamma \xi_{2}(t) - \frac{2 \beta v}{k[1 + v^2/k]^2} \xi_{1}(t - \tau)\right)\mrm{d} t + \frac{1}{\sqrt{N}} \left(\frac{\beta}{1 + v^2/k} + \gamma w\right)^{1/2} \mrm{d} W_{2}(t).
\end{aligned}
\end{equation}

We are now in position to apply the large deviations framework of our  paper to~\eqref{e:LNA}.
Before doing so, we perform a preliminary numerical calculation and comment on the role of trajectory histories.

We numerically compute the stationary points of~\eqref{e:dRRE}. 
We work with the parameter set $\beta = 0.73$, $k = 0.05$, $\gamma = \ln (2)$, and $\tau = 1$, a parameter set for which~\eqref{e:dRRE} has two stable stationary states and one saddle stationary state. 
We find these states by setting the drift expressions in~\eqref{e:dRRE} equal to zero along with $x(t - \tau) = x$ and $y(t - \tau) = y$. Approximate solutions can be found numerically using many well-known iterative methods. The two stable stationary states are approximately $(v,w) \approx (0.0498, 1.0033)$ and $(1.0033, 0.0498)$. The stationary saddle is approximately $(0.3306, 0.3306)$.

Notice that since the Gaussian diffusion~\eqref{e:LNA} contains delay, one must specify a trajectory history over the time interval $[-\tau , 0]$ in order to properly initialize the equation.
Trajectory history will influence the evolution of the mean of the Gaussian diffusion with delay and will therefore affect the computation of optimal large deviations trajectories.
In general, this history may be deterministic or random.
For our current study, we work with deterministic histories and take them to be constant on $[-\tau , 0]$.
See Figure~\ref{fig:Mean} for examples of the evolution of the mean of the Gaussian diffusion using various histories.
Finally, note that although the process $\xi (t)$ is Gaussian, it will not be centered if the history is not identically zero.
To be consistent with the notation of Section~\ref{ss:path-theory}, we write the process that locally approximates the delay chemical Langevin equation as $m(t) + \vep Z_{t}$, where $m(t) = \mbb{E} [\xi (t)]$, $\vep = \frac{1}{\sqrt{N}}$, and $Z_{t}$ satisfies~\eqref{e:LNA} with no small parameter ($N=1$) and history zero.

\subsection{Optimal escape trajectories and exit points - analysis}

We now apply our large deviations framework to the Gaussian diffusion that approximates the delay chemical Langevin equation~\eqref{e:dCLE} near $(v,w)$.
We begin with an analytical view and then follow with numerical simulation.

We find the most likely exit  path with constant initial history $m(0)$ that exits the disk
\begin{equation*}
D = \vset{ (z_{1}, z_{2} ) : (z_{1} - v)^{2} + (z_{2} - w)^{2} \leqs R^{2}}.
\end{equation*}
(We choose $R = 0.3$ for the numerical computations in Section~\ref{ss:Results} so that the neighborhood of $(v,w)$ has size of order one but remains bounded away from the separatrix.)
To find this optimal path, we first find the path of least energy that exits $D$ at a preselected point $q \in \partial D$ and at a preselected time $T$.
We then optimize over $T$ and $q$.
For fixed exit time $T$ and exit point $q \in \partial D$, the optimal escape path and associated energy are given by
\begin{align*}
h(s) &= \rho(s,T)\left[\rho(T,T)^{-1}(q - m(T))\right] + m(s)\\
\lambda_{h}(T, q) &= \frac{1}{2} \left[\rho(T,T)^{-1}(q - m(T))\right] \cdot (q - m(T))
\end{align*}
using~\eqref{e:Xt-minimizer} and~\eqref{e:Xt-minimizer-energy}.
Here, $s$ ranges over $[0,T]$ and $\rho(s,t)$ is the covariance matrix of $Z_{t}$ at times $s,t\in [0,T].$
Note that we are using the terms ``exit time'' and ``escape path'' loosely since we do not impose the {\itshape a priori} condition that $h$ remain inside $D$ until it reaches $q$ at time $T$.

In order to optimize over $q$ and $T$, we first fix $T$ and optimize $\lambda_h(T, q)$ over points $q\in \partial D.$ 
Notice that $\lambda_h(T, q)$ is a classical quadratic form on $\mbb{R}^2$ for fixed $T$, so we apply standard minimization techniques to find the minimizer $\hat{q} (T)$ analytically.
The minimization problem for fixed $T$ is 
\begin{equation}\label{e:MinLNA}
\min_{q} \lambda_{h} (T, q) \;\;\; \mbox{ subject to }\;\;\; (q_1 - v)^2 + (q_2 - w)^2 = R^{2}.
\end{equation}

Using a Lagrange multiplier $\mu\in \mbb{R}$, define the Lagrangian 
\begin{equation*}
L_{\mu}(q) \defas \lambda_{h} (T, q) - \mu ((q_1 - v)^2 + (q_2 - w)^2 - R^{2}).
\end{equation*}
Calculating the gradient $\nabla_{q}(L_{\mu}(q))$ and setting the gradient equal to zero yields the equation
\begin{equation}\label{e:MinEquation}
\rho(T,T)^{-1}(q - m(T)) = 2 \mu (q - (v,w)^{*}).
\end{equation}

Notice that if $m(T) = (v,w)^{*}$, then~\eqref{e:MinEquation} becomes an eigenvalue problem for $\rho (T,T)^{-1}$.
In this case, the optimal exit point $\hat{q} (T)$ is such that $\hat{q} (T) - (v,w)^{*}$ is the eigenvector of $\rho (T,T)^{-1}$ corresponding to the smallest eigenvalue, and the energy of the optimal path that exits $D$ at time $T$ is proportional to this smallest eigenvalue.

This observation has two implications.
First, if the history of the linear noise process is taken to be $m(t) = (v,w)^{*}$ on $[-\tau , 0]$, then we will have $m(t) = (v,w)^{*}$ for all $t \geqs 0$ as well.
In this case, minimizing $\lambda_{h}(T, q)$ over $T$ and $q$ to find the optimal escape time $T_{\text{opt}}$ and the optimal escape point $\hat{q} (T_{\text{opt}})$ amounts to minimizing the smallest eigenvalue of $\rho (T,T)^{-1}$ over $T$.
Since~\eqref{e:LNA} is essentially an Ornstein-Uhlenbeck process with delay, we expect the smallest eigenvalue of $\rho (T,T)^{-1}$ to decrease monotonically toward a limiting value as $T \to \infty $.
See Figure~\ref{fig:Eigenvalues} for numerical evidence.
There exists no minimizer of $\lambda_{h}(T, q)$ in this case, as we would have $T_{\text{opt}} = \infty $.

Second, regardless of the initial history of the linear noise process, $m(T) \to (v,w)^{*}$ as $T \to \infty $ for the parameters we have selected.
Consequently,~\eqref{e:MinEquation} is approximately an eigenvalue problem for large values of $T$, so for such $T$ the optimal exit point $\hat{q} (T)$ will be such that $\hat{q} (T) - (v,w)^{*}$ is close to the eigenvector of $\rho (T,T)^{-1}$ corresponding to the smallest eigenvalue.

\subsection{Numerical results} 
\label{ss:Results}

We compute the optimal path of escape, the optimal exit time $T_{\text{opt}}$, and the optimal exit point $\hat{q} (T_{\text{opt}}) \in \partial D$ for the linear noise process~\eqref{e:LNA} that approximates the toggle switch~\eqref{e:dCLE} in the disk $D$.
Along the way, we discuss interesting related computations.

{\bfseries Parameter selection.}
We set $\beta = 0.73$, $k = 0.05$, $\gamma = \ln (2)$, and $\tau = 1$ for the toggle switch.
System size for the linear noise approximation~\eqref{e:LNA} is $N = 1000$.
The history of the linear noise process is taken to be the constant position $(0.0453, 1.1323)$ over the time interval $[-\tau , 0]$.
We choose $R = 0.3$ for the radius of $D$ so that this neighborhood of $(v,w)$ has size of order one but remains bounded away from the separatrix.

{\bfseries Optimization algorithm.}
To compute the optimal escape path, exit time, and exit point, we execute the following algorithm.
\begin{itemize}[leftmargin=*, labelindent=\parindent, itemsep=0.5ex]
\item Simulate the mean and covariance equations for a sufficiently large $T_{\text{large}}$ using step sizes $\Delta t = \Delta s = \tau / 500$.
\item Discretize the boundary of the disk $D$ using discretization $\Delta r = 0.006$ of $[-R,R]$.
\item For each time $t_{j} = (j-1) \Delta t \in [0,T_{\text{large} }]$ and each point $q_{k}$ on the discretized boundary of the disk, compute the optimal trajectory that exits at time $t_{j}$ through $q_{k}$ as well as the energy $E_{j,k}$ of this trajectory.
\item Minimize over the entries of the matrix $E$ in order to find the optimal exit time and exit point (and hence the overall optimal path of escape).
\end{itemize}

{\bfseries Mean and covariance.}
We first compute the mean and covariance of the linear noise process.
Figure~\ref{fig:Mean} (blue curves) illustrates the evolution of the mean for our parameter set.
As expected, the mean converges to the stationary state $(v,w)$ (moved to $(0,0)$ in Figure~\ref{fig:Mean}).
It is important to choose $T_{\text{large}}$ sufficiently large so that the covariance matrix $\rho (T_{\text{large}}, T_{\text{large}})$ has stabilized and the mean is close to the stationary state.
Fig.~\ref{fig:Variance} and Fig.~\ref{fig:Eigenvalues} provide evidence that this stabilization occurs by time $T = 20$ for our parameter set.
In particular, the variances of the two components of $Z_{t}$ stabilize by time $20$ (Fig.~\ref{fig:Variance}).
Fig.~\ref{fig:Eigenvalues} illustrates that the smallest eigenvalue of $\rho (t,t)^{-1}$ stabilizes as well.

\begin{figure}[ht]
\centering
\begin{subfigure}{.45\textwidth}
\centering
\includegraphics[scale = 0.5]{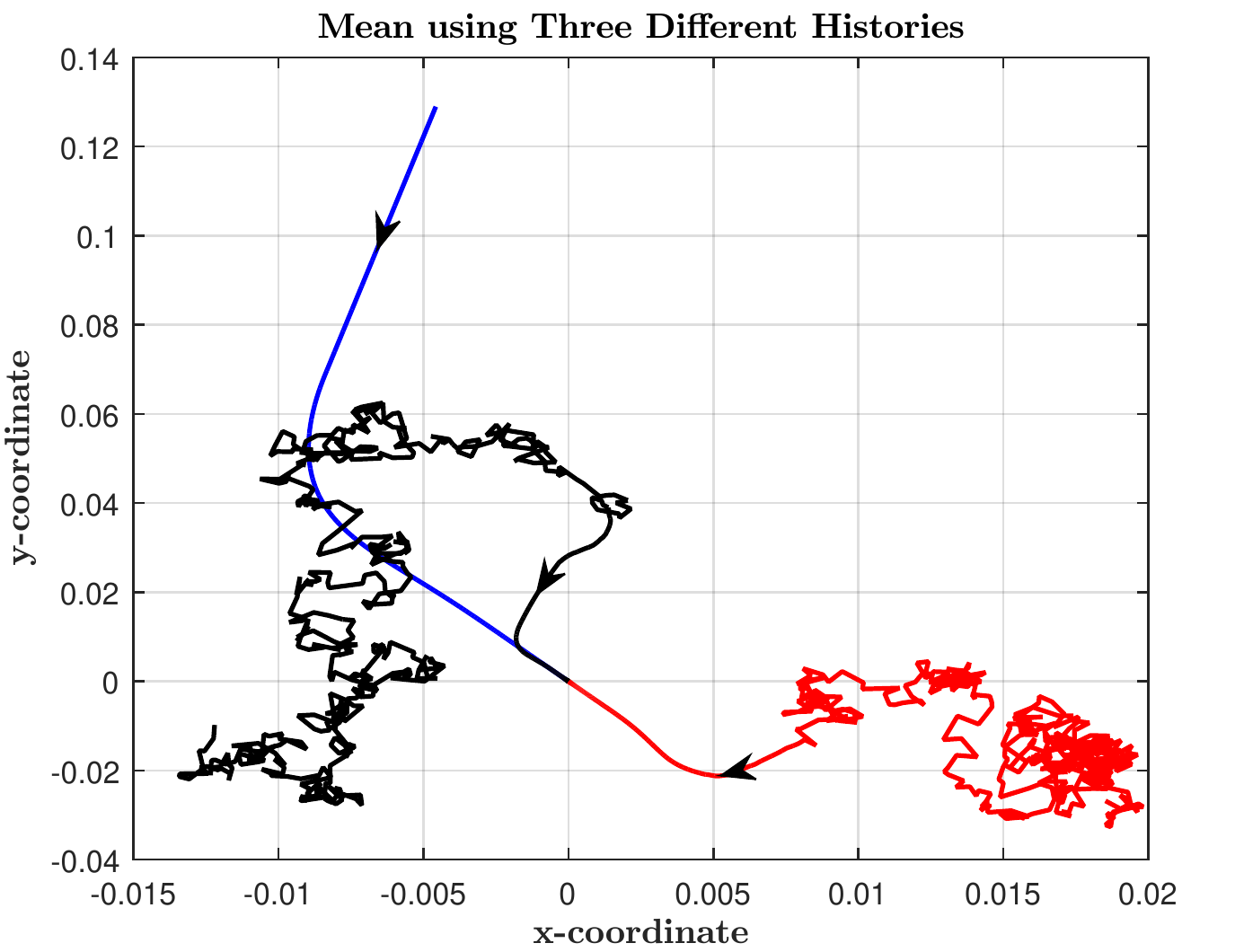}
\caption{}
\label{fig:Example-Means}
\end{subfigure}%
\begin{subfigure}{.45\textwidth}
\centering
\includegraphics[scale = 0.5]{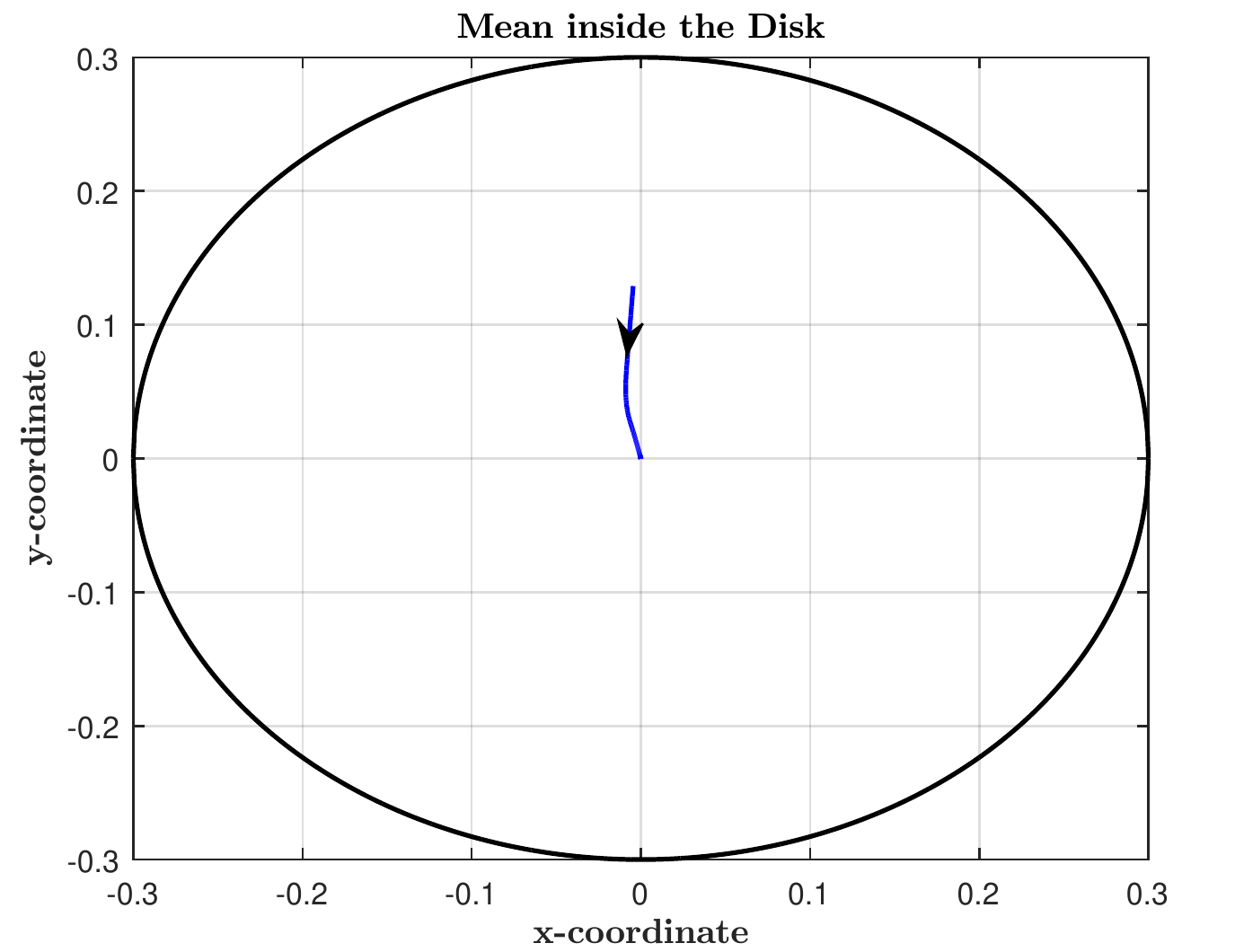}
\caption{}
\label{fig:Mean-Disk}
\end{subfigure}%
\caption{Evolution of the mean of the linear noise process.
Here the stationary state $(v,w)$ has been shifted to the origin.
{\bfseries (\ref{fig:Example-Means})}
Blue curve: evolution of the mean using the constant history $(0.0453, 1.1323)$ (or $(-0.0046, 0.1289)$ in local coordinates).
Red and black curves: evolution of the mean using trajectory segments of~\eqref{e:dCLE} for histories.
In all three cases, the mean converges to the stationary state.
{\bfseries (\ref{fig:Mean-Disk})}
Another view of the blue curve from Fig.~\ref{fig:Example-Means}.}
\label{fig:Mean}
\end{figure}

\begin{figure}[ht]
\centering
\begin{subfigure}{.45\textwidth}
\centering
\includegraphics[scale = 0.5]{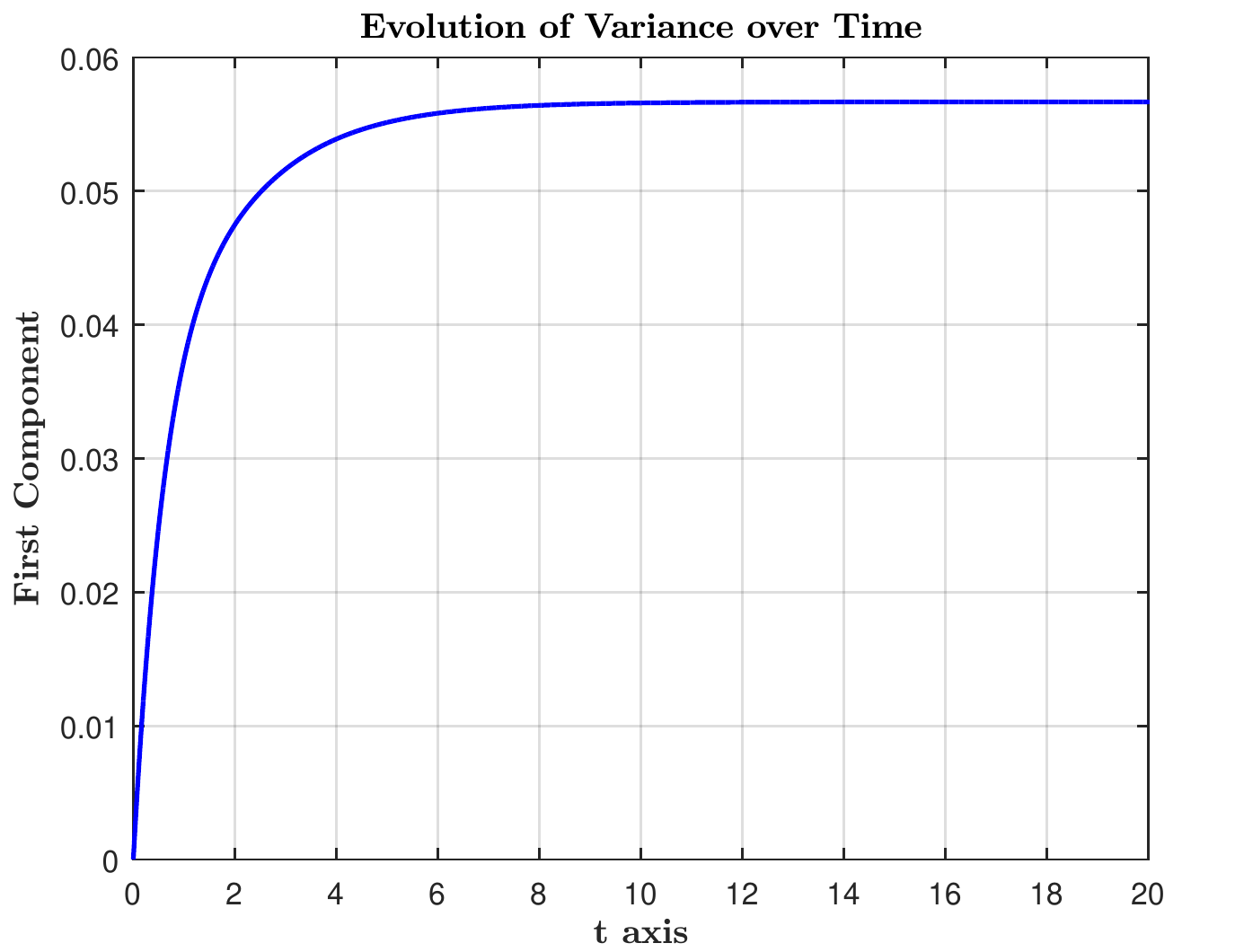}
\caption{}
\label{fig:Variance-time}
\end{subfigure}%
\begin{subfigure}{.45\textwidth}
\centering
\includegraphics[scale = 0.5]{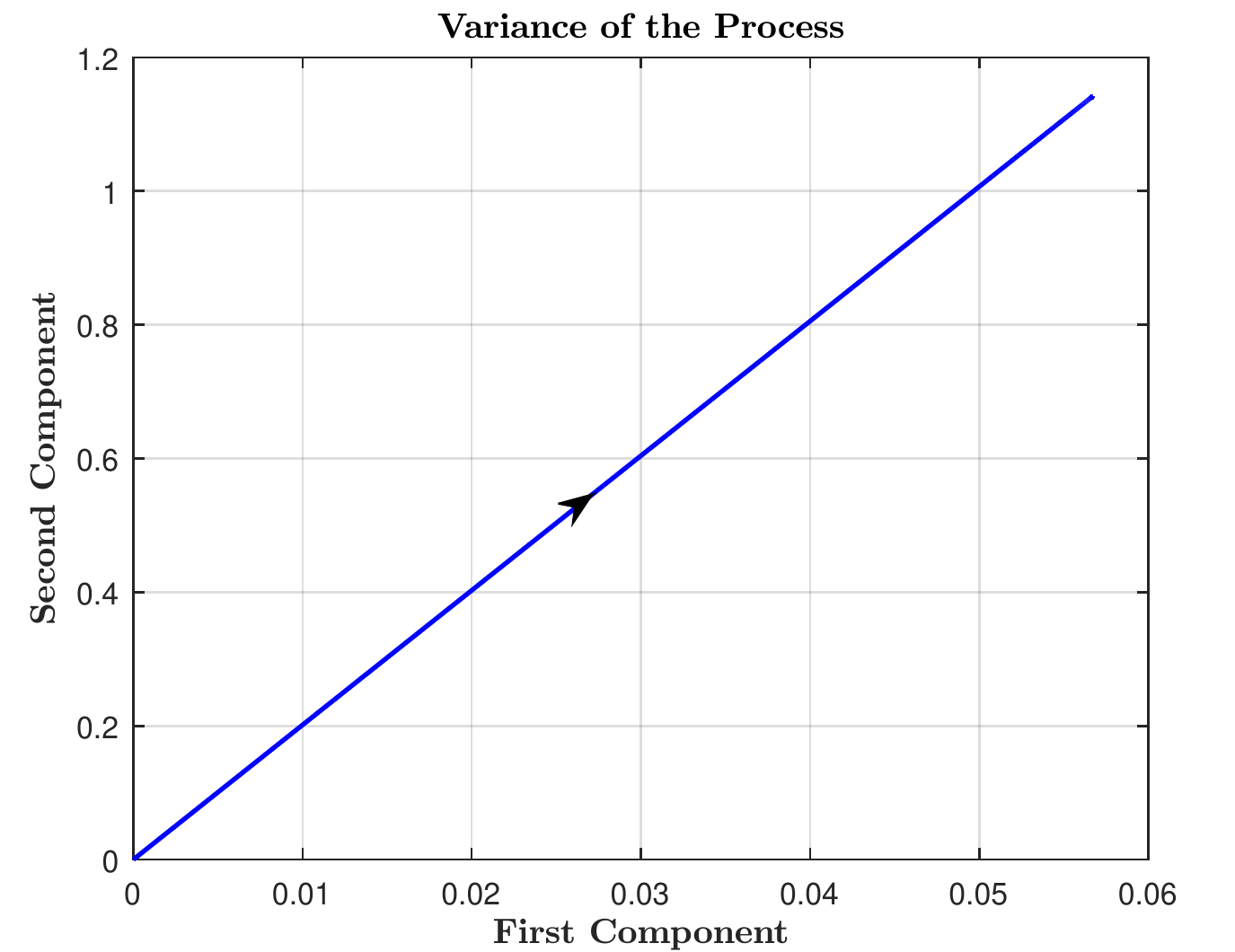}
\caption{}
\label{fig:Variance-comp}
\end{subfigure}%
\caption{Evolution of the variances of the two components of $Z_{t}$ over time.
{\bfseries (\ref{fig:Variance-time})}
The variance of the first component stabilizes to approximately $0.0567$ by time $20$.
{\bfseries (\ref{fig:Variance-comp})}
A linear relationship exists between the evolutions of variances of the first and second components.
By time $20$, the variance of the second component has stabilized to approximately $1.1409$.}
\label{fig:Variance}
\end{figure}

\begin{figure}[ht]
\centering
\includegraphics[scale = 0.5]{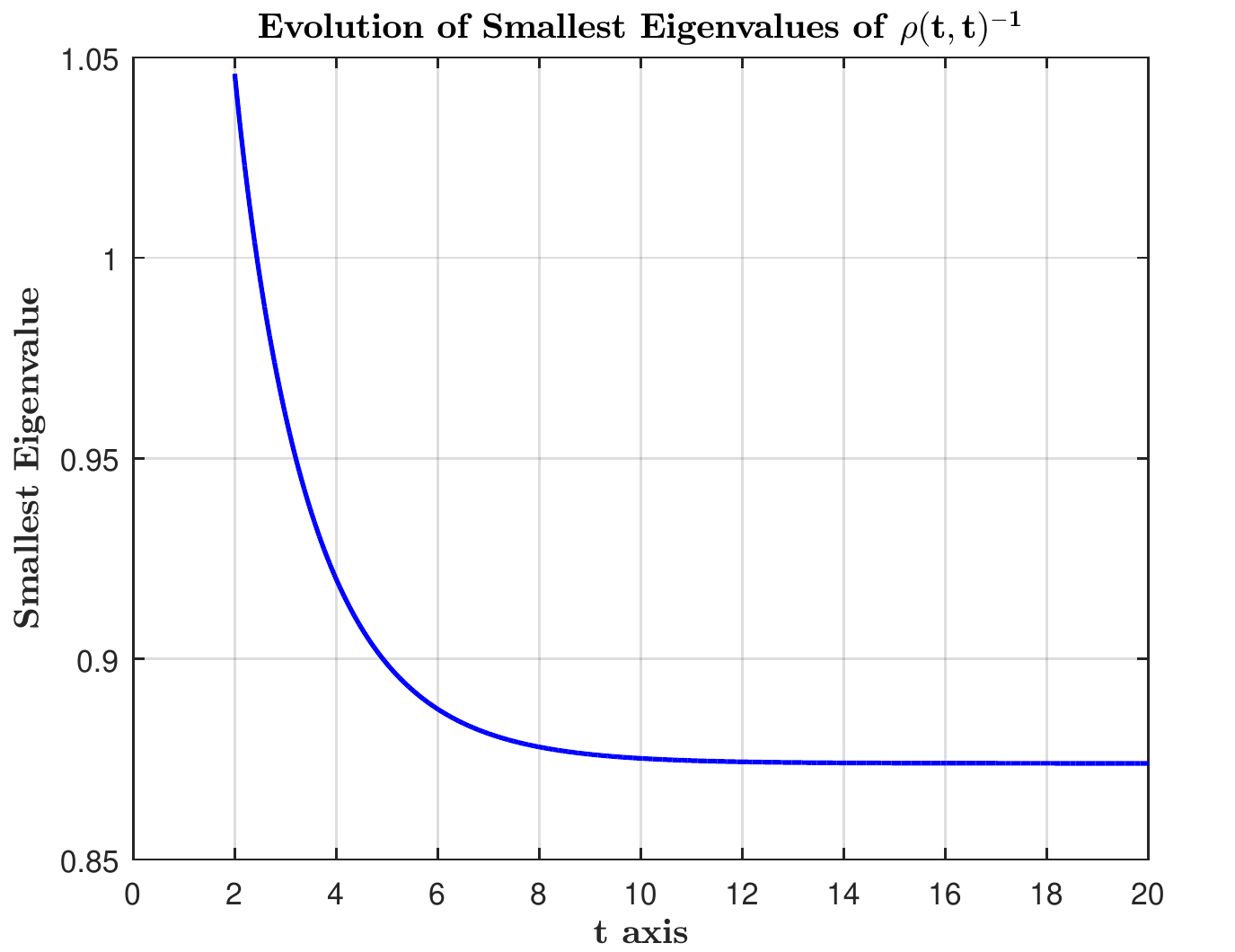}
\caption{The smallest eigenvalue of the covariance matrix $\rho(t,t)^{-1}$ stabilizes to approximately $0.874$ by time $20$.} 
\label{fig:Eigenvalues}
\end{figure}

{\bfseries Numerical optimization results.}
We first examine the behavior of optimal paths and optimal path energies for fixed exit times.
Fig.~\ref{fig:EnergyFuncQ} illustrates the behavior of optimal path energy as a function of exit point over the upper half of $\partial D$ for the fixed exit time $T = 10$.
Note that optimal path energy is minimized near the top of $\partial D$.
Fig.~\ref{fig:OptTrajs} depicts three different optimal escape paths for fixed escape time $T = 20$ and three different exit points.
Notice that these trajectories follow the mean for some time before breaking away toward their respective exits.
This behavior should not occur for the optimal exit time $T_{\text{opt}}$ and the optimal exit point $\hat{q} (T_{\text{opt}})$.
Fig.~\ref{fig:BestTraj} (blue curve) illustrates the overall optimal escape trajectory.
This trajectory exits at time $T_{\text{opt}} = 1.482$ and exit point $\hat{q} (T_{\text{opt}}) = (0.0384, 1.3031)$.
Observe that the overall optimal escape trajectory diverges from the mean immediately.

Fig.~\ref{fig:ExOptTrajs} depicts overall optimal escape trajectories using three different constant initial histories.
Notice that if the initial history is located in the lower half of $D$, then the overall optimal escape trajectory exits through the lower half of $\partial D$.
This happens for the upper half of $D$ as well.
This behavior is natural, since moving `across' the stationary state should not be energetically optimal.
For the initial history corresponding to the blue curve in Fig.~\ref{fig:ExOptTrajs}, the optimal escape path that exits through the bottom half of $\partial D$ does so through $(0.0162, -0.2996)$ (in local coordinates) at exit time $\infty $ with energy $0.0394$.
This energy is strictly larger than that of the blue curve in Fig.~\ref{fig:ExOptTrajs} ($0.0348$).

\newpage

\begin{figure}[ht]
\centering
\begin{subfigure}{.45\textwidth}
\centering
\includegraphics[scale = 0.5]{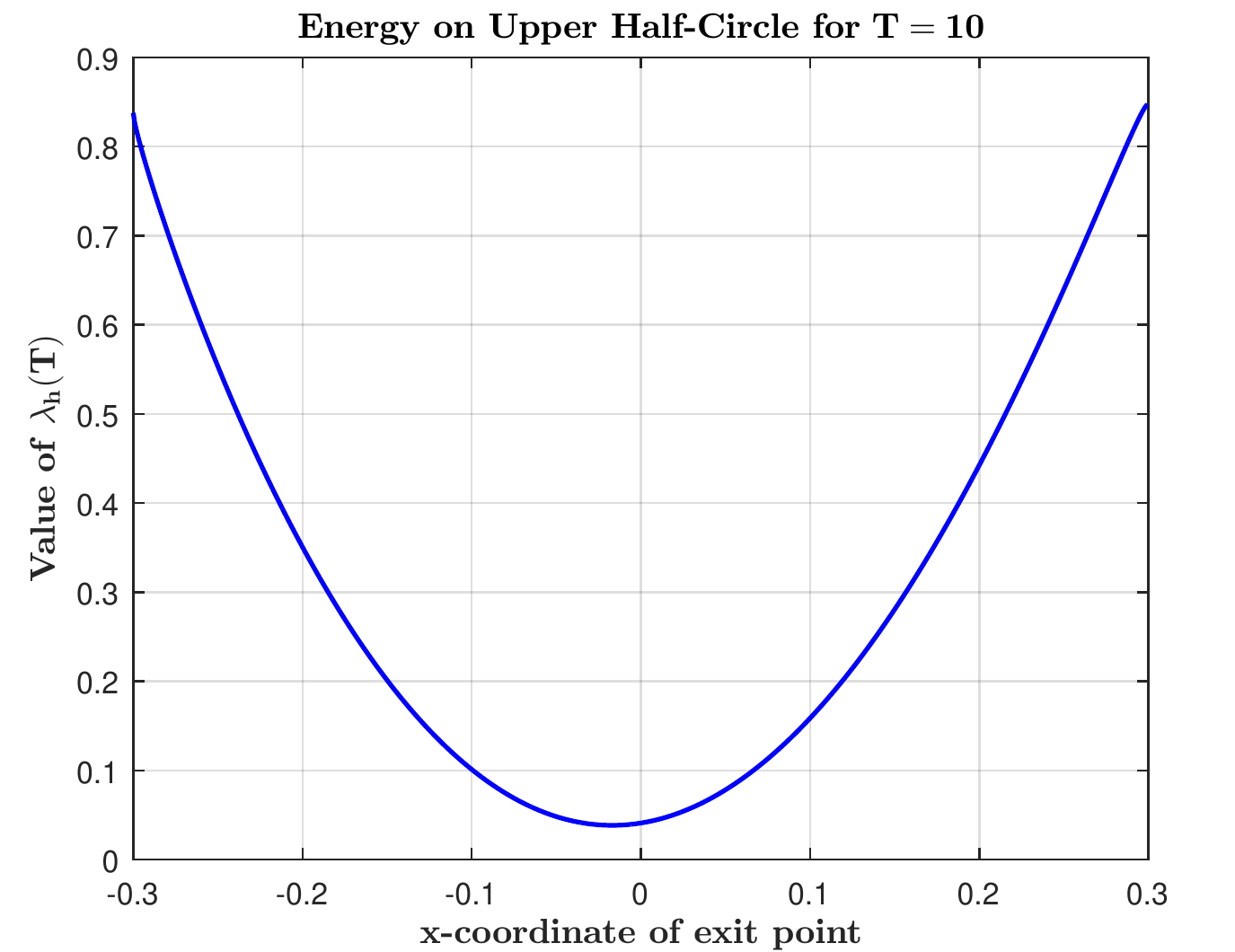}
\caption{}

\label{fig:EnergyFuncQ}
\end{subfigure}%
\begin{subfigure}{.45\textwidth}
\centering
\includegraphics[scale = 0.5]{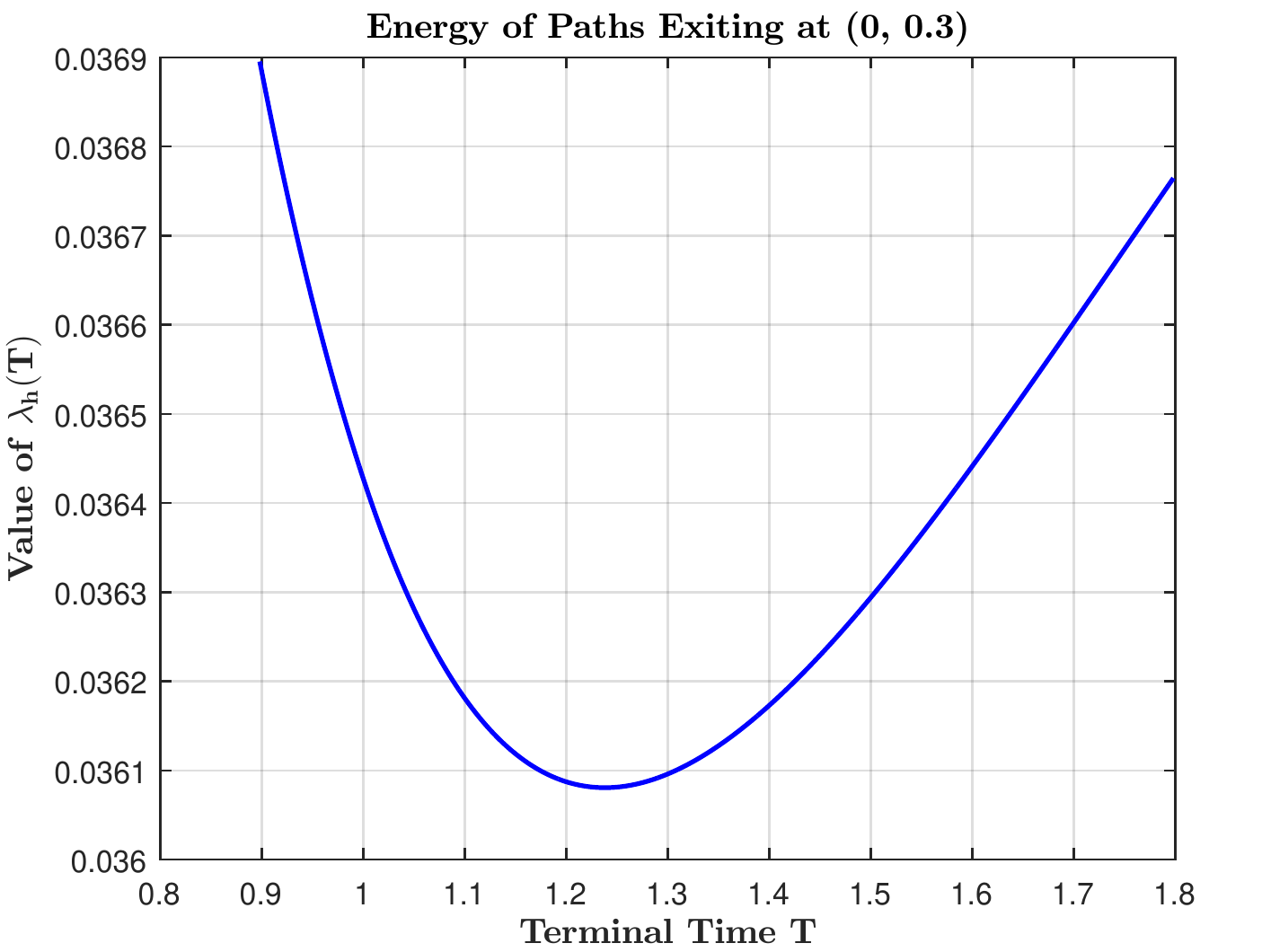}
\caption{}

\label{fig:EnergyFuncT}
\end{subfigure}
\\
\begin{subfigure}{.45\textwidth}
\centering
\includegraphics[scale = 0.5]{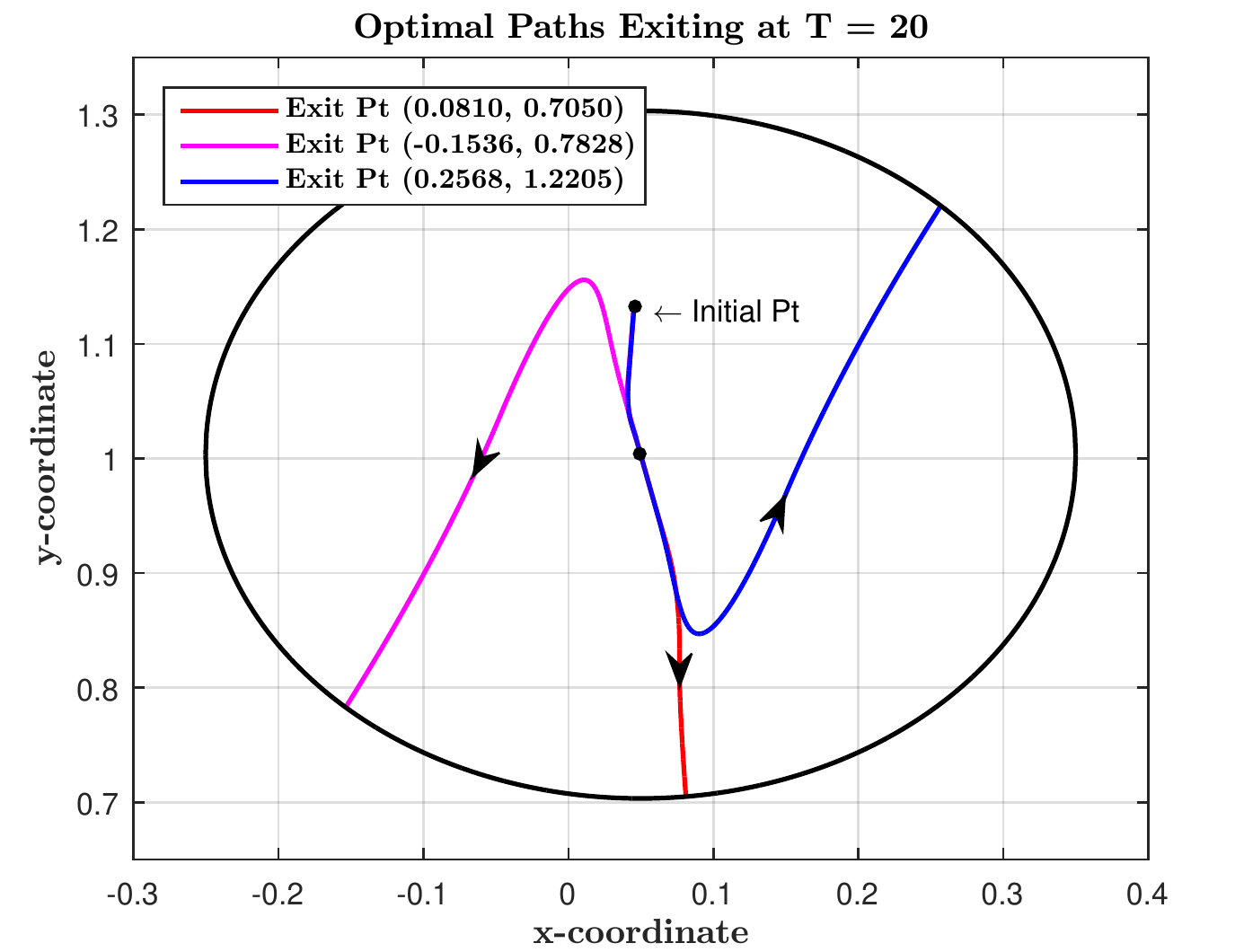}
\caption{}
\label{fig:OptTrajs}
\end{subfigure}%
\begin{subfigure}{.45\textwidth}
\centering
\includegraphics[scale = 0.5]{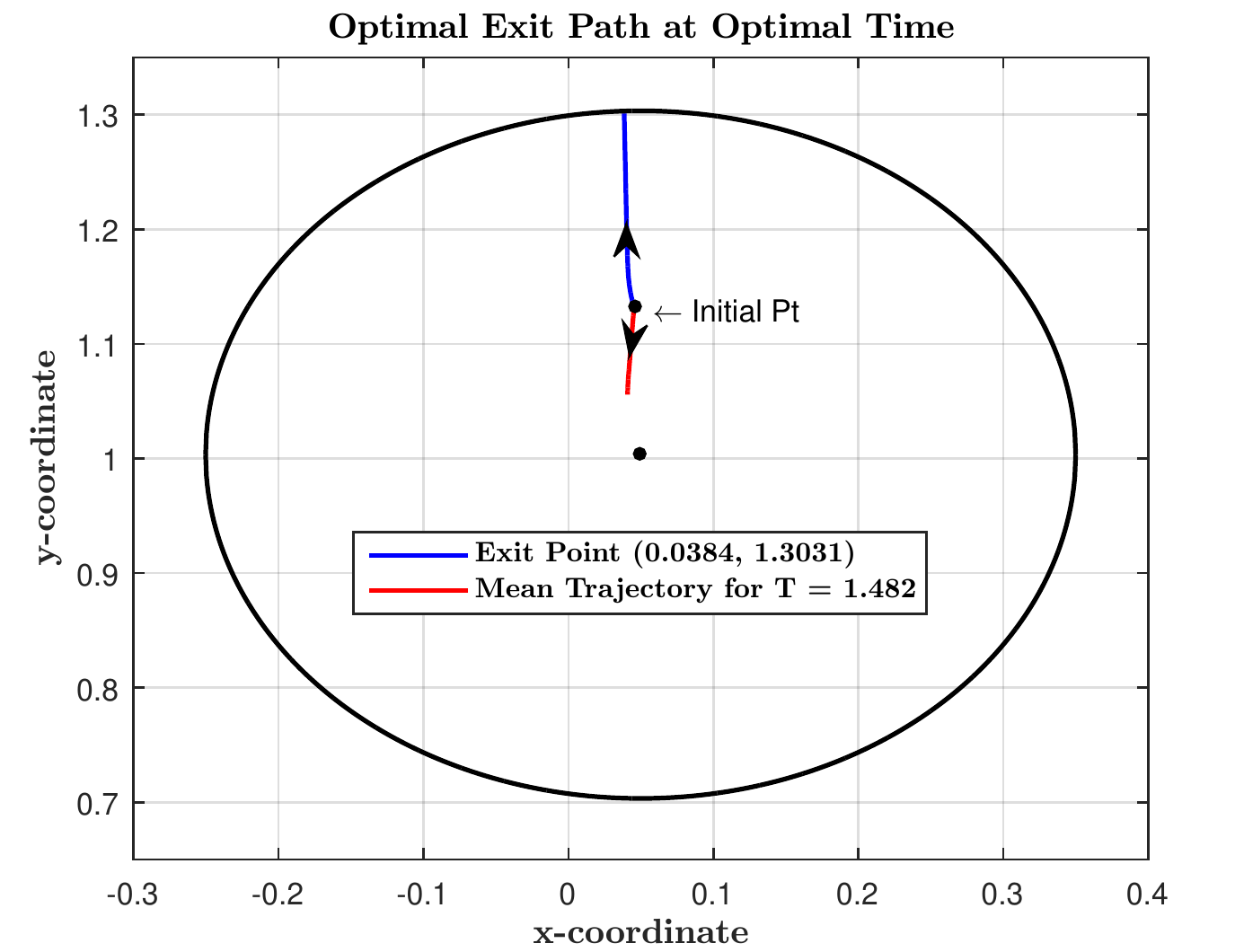}
\caption{}
\label{fig:BestTraj}
\end{subfigure}%
\caption{Optimal exit data.
{\bfseries (\ref{fig:EnergyFuncQ})}
Energy of the optimal exit path at time $T=10$ as a function of chosen exit point on the upper half of $\partial D$.
The energy is minimized near the top of $D$.
{\bfseries (\ref{fig:EnergyFuncT})}
Energy of the optimal exit path as a function of exit time $T$ for fixed exit point $(0,0.3)$ (the top of $D$ in local coordinates).
{\bfseries (\ref{fig:OptTrajs})}
Three different optimal escape paths for fixed escape time $T = 20$ and three different exit points.
Notice that these trajectories follow the mean for some time before breaking away toward their respective exits.
This behavior should not occur for the optimal exit time $T_{\text{opt}}$ and the optimal exit point $\hat{q} (T_{\text{opt}})$.
Energy values associated with the red, magenta, and blue trajectories are $0.0413$, $0.4527$, and $0.4661$, respectively. 
{\bfseries (\ref{fig:BestTraj})}
Overall optimal escape trajectory.
This trajectory exits at time $T_{\text{opt}} = 1.482$ and exit point $\hat{q} (T_{\text{opt}}) = (0.0384, 1.3031)$ with energy $0.0348$.}
\label{fig:Optimization}
\end{figure}

\newpage

\begin{figure}[ht]
\centering
\includegraphics[scale = 0.5]{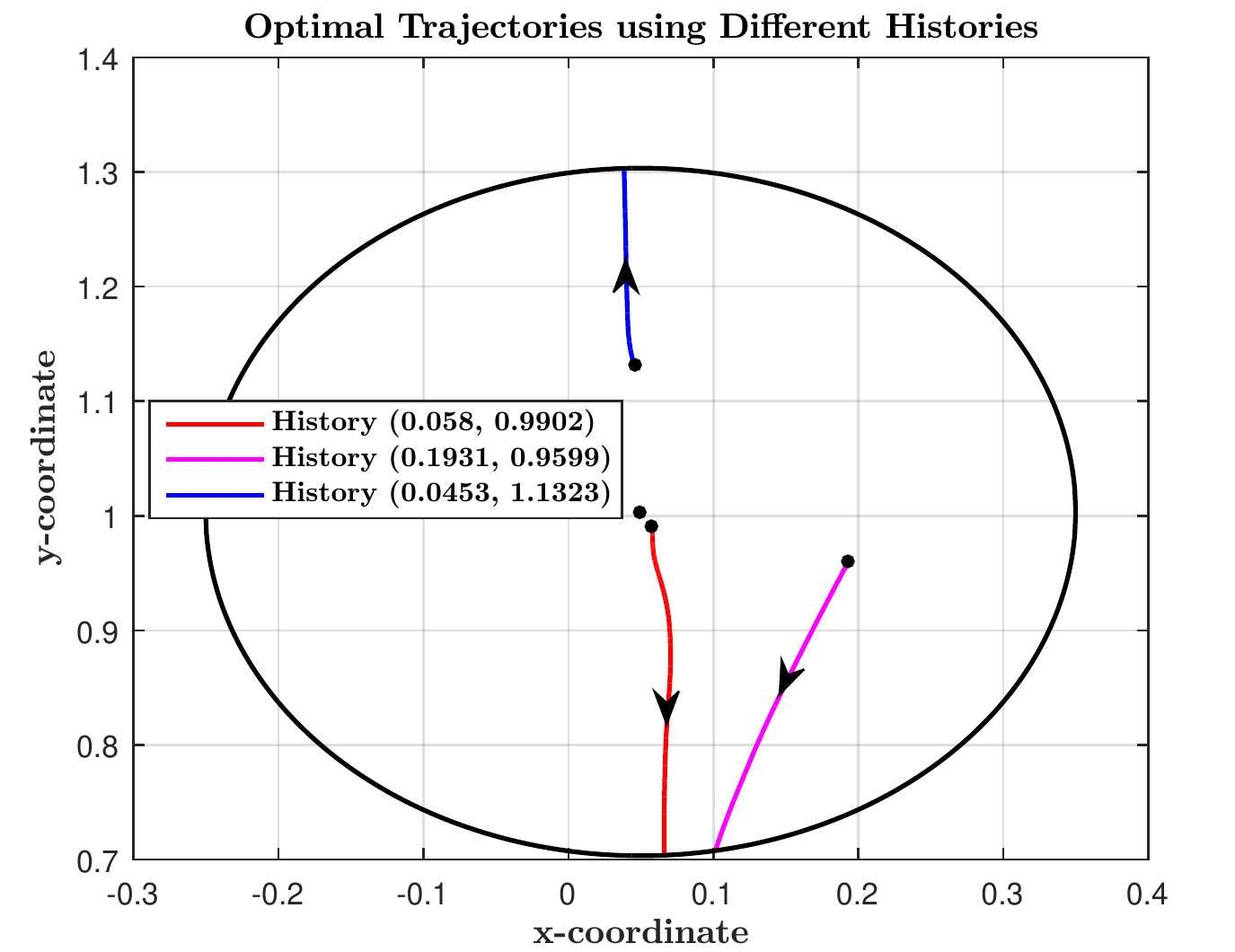}
\caption{Optimal escape trajectories from $D$ for the linear noise process using three different constant initial histories.
Notice that if the initial history is located in the lower half of $D$, then the optimal escape trajectory exits through the lower half of $\partial D$.
This happens for the upper half of $D$ as well.
The energy associated with the red, magenta, and blue trajectories is $0.0389,$ $0.0074,$ and $0.0348$, respectively.}
\label{fig:ExOptTrajs}
\end{figure}


\bibliographystyle{siam}
\bibliography{LDlit-v11,metastability-biochem-lit-v11}

\end{document}